\documentclass{amsart}

\usepackage{txfonts}
\usepackage[all]{xy}
\usepackage{hyperref}
\usepackage[abbrev]{amsrefs}

\theoremstyle{plain}
\newtheorem{theorem}{Theorem}[section]
\newtheorem{proposition}[theorem]{Proposition}
\newtheorem{corollary}[theorem]{Corollary}
\newtheorem{lemma}[theorem]{Lemma}

\theoremstyle{remark}
\newtheorem{remark}[theorem]{Remark}
\newtheorem{example}[theorem]{Example}

\theoremstyle{definition}
\newtheorem{definition}[theorem]{Definition}
\newtheorem{notation}[theorem]{Notation}

\newcommand{\F}{\mathbb{F}}
\newcommand{\N}{\mathbb{N}}
\newcommand{\Z}{\mathbb{Z}}

\newcommand{\cat}[1]{\mathbf{\mathcal{#1}}} 
\newcommand{\ind}[1]{\widetilde{#1}} 
\newcommand{\ord}[1]{\mathbf{\underline{#1}}}

\newcommand{\al}{\alpha}

\newcommand{\De}{\Delta}
\newcommand{\ep}{\epsilon}
\newcommand{\io}{\iota}
\newcommand{\ka}{\kappa}
\newcommand{\Om}{\Omega}

\newcommand{\si}{\sigma}
\newcommand{\Si}{\Sigma}

\newcommand{\rla}{\rightleftarrows}
\newcommand{\rra}{\rightrightarrows}
\newcommand{\ral}{\xrightarrow} 
\newcommand{\Ra}{\Rightarrow}
\newcommand{\La}{\Leftarrow}

\newcommand{\inj}{\hookrightarrow}
\newcommand{\surj}{\twoheadrightarrow}
\newcommand{\bu}{\bullet}
\newcommand{\tild}{\widetilde}

\newcommand{\dfn}{\coloneqq}
\newcommand{\lan}{\left\langle}
\newcommand{\lt}{\ltimes} 

\newcommand{\op}{\oplus}
\newcommand{\ot}{\otimes}
\newcommand{\push}{\sharp}
\newcommand{\ran}{\right\rangle}
\newcommand{\We}{\bigvee}
\newcommand{\x}{\times}

\newcommand{\Ab}{\mathbf{Ab}}
\newcommand{\Alg}{\mathbf{Alg}}
\newcommand{\Bimod}[1]{#1\text{-}\mathbf{Bimod}}
\newcommand{\Cat}{\mathbf{Cat}}
\newcommand{\Com}{\mathbf{Com}}
\newcommand{\Gp}{\mathbf{Gp}}
\newcommand{\GrSet}{\mathbf{GrSet}}
\newcommand{\Ho}{\mathbf{Ho}}
\newcommand{\Hy}{H}
\newcommand{\LL}{\mathbf{L}}
\newcommand{\Mod}{\mathbf{Mod}}
\newcommand{\Modd}[1]{#1\text{-}\mathbf{Mod}}
\newcommand{\Model}{\mathbf{Model}}
\newcommand{\PI}{\mathbf{\Pi}}
\newcommand{\PiAlg}{\mathbf{\Pi Alg}}
\newcommand{\Poset}{\mathbf{Poset}}
\newcommand{\redCom}{\mathbf{Com}^{\mathrm{red}}}
\newcommand{\Set}{\mathbf{Set}}
\newcommand{\tfAb}{\mathbf{Ab}^{\mathrm{tf}}}

\newcommand{\Def}{\textbf}

\DeclareMathOperator*{\colim}{colim}
\DeclareMathOperator{\Ext}{Ext}
\DeclareMathOperator{\Fun}{Fun}
\DeclareMathOperator{\HA}{HA}
\DeclareMathOperator{\HH}{HH}
\DeclareMathOperator{\Hom}{Hom}
\DeclareMathOperator{\HOM}{\underline{Hom}}
\DeclareMathOperator{\HQ}{HQ}
\DeclareMathOperator{\Nil}{Nil}
\DeclareMathOperator{\Src}{Src}
\DeclareMathOperator{\Tor}{Tor}
\DeclareMathOperator{\Triv}{Triv}

\newcommand{\ab}{\mathrm{ab}}
\newcommand{\id}{\mathrm{id}}
\newcommand{\opp}{\mathrm{op}}

\begin{document}

\title{Behavior of Quillen (co)homology with respect to adjunctions} 
\date{\today}


\author{Martin Frankland}             
\email{mfrankla@uwo.ca}
\address{Department of Mathematics,
         University of Western Ontario,
         Middlesex College,
         London, ON, N6A 5B7,
         Canada}

\thanks{Supported in part by an NSERC Postgraduate Scholarship and an FQRNT Doctoral Research Scholarship.}

\thanks{Most of this paper is part of my doctoral work at MIT under the supervision of Haynes Miller, whom I thank heartily for all his support. I also thank Michael Barr, Jacob Lurie, David Blanc, and Dan Christensen for fruitful conversations, as well as the referee for many helpful comments.}




\subjclass[2010]{Primary: 18G55; Secondary: 55U35, 18G30, 13D03.}

\keywords{Quillen, homology, cohomology, adjunction, simplicial, model structure, comparison.}

\begin{abstract}
This paper aims to answer the following question: Given an adjunction between two categories, how is Quillen (co)homology in one category related to that in the other? We identify the induced comparison diagram, giving necessary and sufficient conditions for it to arise, and describe the various comparison maps. Examples are given. Along the way, we clarify some categorical assumptions underlying Quillen (co)homology: cocomplete categories with a set of small projective generators provide a convenient setup.
\end{abstract}

\maketitle

\section{Introduction}

\subsection{Motivation and goals}

Quillen \cite{Quillen67}*{\S II.5} introduced a notion of cohomology that makes use of homotopical algebra and simplicial methods to take derived functors in a non-abelian context, generalizing the derived functors of homological algebra. One of the goals was to solve problems in algebra using methods from homotopy theory, although Quillen cohomology later found many applications to homotopy theory and topology \cite{Goerss07}*{Remark 4.35}.

Quillen cohomology works in a broad context which includes many interesting categories. The case of commutative algebras, the celebrated Andr\'e-Quillen cohomology \cite{Quillen70}*{\S 4} \cite{Andre74} \cite{Goerss07}*{\S 4.4}, was one of the first examples studied. The analogue for associative algebras \cite{Quillen70}*{\S 3} is related to another well studied theory, namely Hochschild cohomology. Quillen exhibited relations between the two \cite{Quillen70}*{\S 8}, which can be useful when cohomology is easier to compute in one category or the other.

This paper investigates the question of relating Quillen (co)homology in different categories, more specifically when two categories are related by an adjunction. Our motivating example was to compute some Quillen cohomology groups of truncated $\Pi$-algebras controlling the obtructions to realization \cite{Blanc04}, which is done in Section \ref{sec:TrunPiAlg}. However, the broader question seems natural, given that adjoint pairs abound in nature.

\subsection{Organization and results}

Section \ref{sec:Kinds} contains some background material. Section \ref{sec:SetupHQ} clarifies the categorical assumptions underlying Quillen cohomology. It consists mostly of category theory, for the purposes of homotopical algebra. The main clarifications are Propositions \ref{AbExist} and \ref{GoodSetupHQ}. Examples \ref{NonAbelBeck} and \ref{NonExact} clarify conditions related to Beck modules being abelian.

Section \ref{sec:Effect} is the heart of the paper, describing the effect of an adjunction on Quillen (co)homology. We first describe the comparison diagram consisting of Quillen pairs, and work out various comparison maps from it. The main result is Theorem \ref{ComparDiag}, from which \ref{EffectQuilHomolC} and \ref{EffectQuilHomolD} follow.

Section \ref{sec:Examples} studies examples of adjunctions where the right adjoint is the inclusion of a regular-epireflective full subcategory. In other words, the right adjoint forgets certain conditions satisfied by the objects, and the left adjoint is the quotient that freely imposes the conditions. The main results are \ref{PnCotCpx} and \ref{HQtrunc2}.

\subsection{Notations, conventions, and terminology}

\subsubsection*{Simplicial objects}

\begin{notation}
Let $\De$ denote the simplicial indexing category, whose objects are the finite ordinals $\ord{n} = \{ 0, 1, \ldots, n \}$, for $n \geq 0$, and maps $\phi \colon \colon \ord{m} \to \ord{n}$ are order-preserving functions.

Given a category $\cat{C}$, denote by $s\cat{C}$ the category of simplicial objects in $\cat{C}$, i.e., the functor category $\Fun(\De^{\opp},\cat{C})$. Denote a simplicial object by $X_{\bu}$, where $X_n \dfn X(\ord{n})$ denotes the object in simplicial degree $n$.

Denote the standard $n$-simplex by $\De^n \dfn \Hom_{\De}(-,\ord{n})$, which is a simplicial set.

We say that a property of a simplicial object $X_{\bu}$ holds \emph{degreewise} if it holds for $X_n$ for all $n \geq 0$, and likewise for maps $f_{\bu} \colon X_{\bu} \to Y_{\bu}$. 

An object $X$ of $\cat{C}$ can be viewed as a constant simplicial object, denoted $X_c$ (or just $X$, by abuse of notation). This defines a fully faithful functor $(-)_c \colon \cat{C} \to s\cat{C}$.
\end{notation}

\subsubsection*{Quillen (co)homology}

\begin{definition}
For an object $X$ of $\cat{C}$, the category of \Def{Beck modules} over $X$ is the category $(\cat{C}/X)_{\ab}$ of abelian group objects in the slice category $\cat{C}/X$. We sometimes use the notation $\Mod_X \dfn (\cat{C}/X)_{\ab}$.
\end{definition}

\begin{definition}
If the forgetful functor $U_X \colon (\cat{C}/X)_{\ab} \to \cat{C}/X$ has a left adjoint $Ab_X \colon \cat{C}/X \to (\cat{C}/X)_{\ab}$, the latter is called \Def{abelianization} over $X$.
\end{definition}

\begin{definition}
For a map $f \colon X \to Y$ in $\cat{C}$, the \Def{direct image} functor $f_! \colon \cat{C}/X \to \cat{C}/Y$ is postcomposition by $f$, which is left adjoint to the \Def{pullback} functor $f^* \colon \cat{C}/Y \to \cat{C}/X$. Since $f^*$ preserves limits, it induces a functor $f^* \colon (\cat{C}/Y)_{\ab} \to (\cat{C}/X)_{\ab}$ also called \Def{pullback}. The \Def{pushforward} along $f$ is the left adjoint $f_{\push} \colon (\cat{C}/X)_{\ab} \to (\cat{C}/Y)_{\ab}$ of $f^*$, if it exists.
\end{definition}

\begin{definition}
The \Def{cotangent complex} $\LL_X$ of $X$ is the derived abelianization of $X$, i.e., the simplicial module over $X$ given by $\LL_X \dfn Ab_X(C_{\bu} \to X)$, where $C_{\bu} \to X$ is a cofibrant replacement of $X$ in $s\cat{C}$.
\end{definition}

\begin{definition}
The $n^{\text{th}}$ \Def{Quillen homology} group of $X$ is the $n^{\text{th}}$ derived functor of abelianization, given by $\HQ_n(X) \dfn \pi_n(\LL_X)$. If the category $\Mod_X$ has a good notion of tensor product $\ot$, then Quillen homology with coefficients in a module $M$ over $X$ is $\HQ_n(X;M) \dfn \pi_n(\LL_X \ot M)$.
\end{definition}

\begin{definition}
The $n^{\text{th}}$ \Def{Quillen cohomology} group of $X$ with coefficients in a module $M$ is the $n^{\text{th}}$ (simplicially) derived functor of derivations, given by $\HQ^n(X;M) \dfn \pi^n \Hom(\LL_X,M)$.
\end{definition}

\begin{definition}
The $n^{\text{th}}$ \Def{abelian cohomology} group of $X$ with coefficients in a module $M$ is the $n^{\text{th}}$ derived functor of derivations in the sense of homological algebra, given by $\HA^n(X;M) \dfn \Ext^n(Ab_X X, M)$. The $n^{\text{th}}$ \Def{abelian homology} group of $X$ with coefficients in $M$ is $\HA_n(X;M) \dfn \Tor_n(Ab_X X, M)$. They can be viewed as abelian approximations of Quillen (co)homology, with comparison maps $\HA^n(X;M) \to \HQ^n(X;M)$ and $\HQ_n(X;M) \to \HA_n(X;M)$.
\end{definition}

\subsubsection*{Category theory}

We follow mostly \cite{Borceux94v1}*{Chapter 4}, \cite{Borceux94v2}*{Chapter 2}, \cite{Quillen67}*{\S II.4}, \cite{Barr02}*{Chapter 1}, and \cite{Adamek94}*{Chapter 1}. All categories will be assumed locally small, i.e., $\Hom_{\cat{C}}(X,Y)$ forms a set for any two objects $X$ and $Y$ of the category $\cat{C}$. We will not distinguish between the notions of small category and essentially small category, i.e., one that is equivalent to a small category. The term ``finite products'' will always include the nullary product, i.e., the terminal object; likewise for finite coproducts.

\begin{definition}
The \Def{kernel pair} of a map $f \colon X \to Y$ is the pullback of $f$ along itself:
\[
\xymatrix{
X \x_Y X \ar[d] \ar[r] & X \ar[d]^{f} \\
X \ar[r]_{f} & Y \\
}
\]
equipped with its two projections $X \x_Y X \rra X$.
\end{definition}

\begin{definition}
A map $f \colon X \to Y$ is called a \Def{regular epimorphism} if it is the coequalizer of some parallel pair of maps $W \rra X$. Note that a regular epimorphism is indeed automatically an epimorphism. 
\end{definition}

\begin{remark}
A map $f \colon X \to Y$ is called an \emph{effective epimorphism} if $f$ admits a kernel pair $X \x_Y X \rra X$ and $f$ is the coequalizer of its kernel pair. If a map $f$ is a coequalizer of some parallel pair (i.e., a regular epimorphism) and $f$ admits a kernel pair, then $f$ is the coequalizer of its kernel pair \cite{Borceux94v1}*{Proposition 2.5.7}. Thus, in a category $\cat{C}$ with all kernel pairs, the notions of regular epimorphism and effective epimorphism coincide.
\end{remark}

\begin{definition}
A category $\cat{C}$ is called \Def{regular} if it satisfies the following conditions \cite{Borceux94v2}*{Definition 2.1.1}.
\begin{itemize}
\item Every map $f \colon X \to Y$ has a kernel pair.
\item Every kernel pair $X \x_Y X \rra X$ has a coequalizer.
\item The pullback of a regular epimorphism along any map exists and is again a regular epimorphism.
\end{itemize}
\end{definition}

\begin{remark}
The notion of regular category in \cite{Barr02}*{Chapter 1 \S 8.9} is more restrictive. The first two conditions are strengthened to the existence of all finite limits and of all coequalizers.
\end{remark}

\begin{definition}
An object $X$ of $\cat{C}$ is called \Def{finitely presentable} if the functor \linebreak $\Hom_{\cat{C}}(X,-) \colon \cat{C} \to \Set$ preserves filtered colimits \cite{Adamek94}*{Definition 1.1}.

Given a regular cardinal $\ka$, an object $X$ is called \Def{$\ka$-presentable} if the functor \linebreak $\Hom_{\cat{C}}(X,-) \colon \cat{C} \to \Set$ preserves $\ka$-filtered colimits \cite{Borceux94v2}*{Definition 5.1.1} \cite{Adamek94}*{Definition 1.13}. An object is \Def{presentable} if it is $\ka$-presentable for some regular cardinal $\ka$. Finitely presentable is thus another name for $\aleph_0$-presentable.
\end{definition}

\begin{remark}
This notion of finitely presentable is called \emph{small} in \cite{Quillen67}*{\S II.4}. Let us clarify the distinction between smallness and presentability.

Given a regular cardinal $\ka$, an object $X$ is called \emph{$\ka$-small} if the functor $\Hom_{\cat{C}}(X,-) \colon \cat{C} \to \Set$ preserves $\ka$-directed sequential colimits, where \emph{sequential} means indexed by a well-ordered set. An object is \emph{small} if it is $\ka$-small for some regular cardinal $\ka$. Smallness and various weaker conditions play a role in the small object argument; cf. \cite{Goerss99}*{Definition II.6.1}, \cite{Hovey99}*{Definition 2.1.3}, and \cite{Hirschhorn03}*{Definition 10.4.1}.

By definition, every $\ka$-presentable object is also $\ka$-small. For $\ka = \aleph_0$, the converse holds: an object is $\aleph_0$-presentable if and only if it is $\aleph_0$-small \cite{Adamek94}*{Corollary 1.7}. However, this is no longer true for larger cardinals. For example, \cite{Adamek94}*{Remark 1.21} describes an object which is $\aleph_1$-small but not $\aleph_1$-presentable.
\end{remark}

\begin{definition}
A class $\cat{G}$ of objects of $\cat{C}$ is called a \Def{class of generators} if for every object $X$ of $\cat{C}$, there is a regular epimorphism $\amalg_i G_i \surj X$ from a coproduct of objects $G_i$ in $\cat{G}$. This notion is also called a \emph{regular class of generators} in the literature \cite{Borceux94v1}*{Definition 4.5.3}.
\end{definition}

\begin{definition}
A category $\cat{C}$ is \Def{locally finitely presentable} if it is cocomplete and has a set $\cat{G}$ of finitely presentable objects such that every object of $\cat{C}$ is a filtered colimit of objects from $\cat{G}$ \cite{Adamek94}*{Definition 1.9}.

Given a regular cardinal $\ka$, a category $\cat{C}$ is \Def{locally $\ka$-presentable} if it is cocomplete and has a set $\cat{G}$ of $\ka$-presentable objects such that every object of $\cat{C}$ is a $\ka$-filtered colimit of objects from $\cat{G}$ \cite{Adamek94}*{Definition 1.17} \cite{Borceux94v2}*{Definition 5.2.1}. A category is \Def{locally presentable} if it is locally $\ka$-presentable for some regular cardinal $\ka$. Locally finitely presentable is thus another name for locally $\aleph_0$-presentable.
\end{definition}

\begin{definition} \label{def:Proj}
An object $P$ of $\cat{C}$ is called \Def{projective} if the functor $\Hom_{\cat{C}}(P,-) \colon \cat{C} \to \Set$ preserves regular epimorphisms. More explicitly, recall that (regular) epimorphisms in $\Set$ are the surjections, so that an object $P$ of $\cat{C}$ is projective if and only if maps out of $P$ can be lifted across any regular epimorphism $f \colon X \surj Y$, as illustrated in the diagram:
\begin{equation} \label{eq:ProjLift}
\xymatrix{
& P \ar@{-->}[dl] \ar[d] \\
X \ar@{->>}[r]_{f} & Y. \\
}
\end{equation}
The category $\cat{C}$ \Def{has enough projectives} if for every object $X$ of $\cat{C}$, there exists a regular epimorphism $P \surj X$ from a projective object $P$.
\end{definition}

\begin{remark}
This notion of projective is also called \emph{regular projective} in the literature \cite{Adamek94}*{Remark 3.4 (5)}. We drop the adjective because regular projectives are the only kind of projectives we work with in this paper. Note that this notion is implied by the notion of projective in \cite{Borceux94v1}*{Definition 4.6.1}; the two notions agree if the category $\cat{C}$ is regular \cite{Borceux94v2}*{Proposition 2.1.4}.
\end{remark}

\begin{remark}
In an abelian category $\cat{A}$, every epimorphism is regular, and the notion of (regular) projective is the usual notion of projective from homological algebra.
\end{remark}

\begin{definition}
Let $\cat{P}$ be a class of objects in $\cat{C}$. A map $f \colon X \to Y$ is \Def{$\cat{P}$-epic} if for every object $P$ of $\cat{P}$, the map $f_* \colon \Hom_{\cat{C}}(P,X) \to \Hom_{\cat{C}}(P,Y)$ is surjective \cite{Christensen02}*{\S 1.1}. In other words, maps out of any $P$ in $\cat{P}$ can be lifted across $f \colon X \surj Y$, as in the diagram \eqref{eq:ProjLift}.

In this paper, $\cat{P}$ will be the class of regular projectives unless otherwise noted. In that case, every regular epimorphism is $\cat{P}$-epic by definition, though a $\cat{P}$-epic map need not be a regular epimorphism. Note morever that projectives and $\cat{P}$-epic maps determine each other via the lifting condition.
\end{definition}

\begin{definition}
A \Def{subobject} of an object $X$ in a category $\cat{C}$ is an equivalence class of monomorphisms $Z \inj X$, up to isomorphism over $X$ \cite{Borceux94v1}*{Definition 4.1.1} . The equivalence class of $Z \inj X$ is denoted $[Z \inj X]$.
\end{definition}


\begin{definition}
A \Def{relation} on an object $X$ is a subobject $[R \inj X \x X]$ \cite{Borceux94v2}*{Definition 2.5.1}.

By abuse of notation, we sometimes blur the distinction between a relation $[R \inj X \x X]$ and one of its representative monomorphisms $R \inj X \x X$.

Note that a monomorphism $R \inj X \x X$ is the same as a pair of maps $r_1, r_2 \colon R \rra X$ which are jointly monomorphic. The \emph{coequalizer} of a relation $[R \to X \x X]$ is defined as the coequalizer of such a pair $R \rra X$ for any  representative map $R \inj X \x X$. The coequalizer is independent of the choice of representative.
\end{definition}

\begin{definition}
An \Def{equivalence relation} $[R \inj X \x X]$ on an object $X$ is a relation on $X$ which satisfies reflexivity, symmetry, and transitivity; see \cite{Borceux94v2}*{Proposition 2.5.5} for more details.
\end{definition}

\begin{definition}
An equivalence relation $R \rra X$ on an object $X$ is called \Def{effective} if it is the kernel pair of some map \cite{Borceux94v2}*{Definition 2.5.3} \cite{Barr02}*{Chapter 1 \S 8.11}.
\end{definition}

\begin{definition}
An \Def{exact category} (in the sense of Barr) is a regular category in which all equivalence relations are effective \cite{Borceux94v2}*{Definition 2.6.1} \cite{Barr02}*{Chapter 1 \S 8.11}.
\end{definition}

\section{What kinds of categories?} \label{sec:Kinds}

In the classic \cite{Quillen67}*{\S II.4}, Quillen introduces a standard simplicial model structure on the category $s\cat{C}$ of simplicial objects in $\cat{C}$, assuming $\cat{C}$ is nice enough. In this section, we recall some conditions on $\cat{C}$ that make this construction work, and we describe the kinds of categories we will be working with.

\subsection{Standard model structure on simplicial objects}

\begin{definition} \label{NiceSplObj}
A complete and cocomplete category $\cat{C}$ \Def{has nice simplicial objects} if the following notions define a (closed) model structure on $s\cat{C}$. A map $f_{\bu} \colon X_{\bu} \to Y_{\bu}$ in $s\cat{C}$ is a:
\begin{itemize}
\item \emph{fibration} (resp. \emph{weak equivalence}) if for every projective object $P$ of $\cat{C}$, the map:
\[
\xymatrix{
\Hom_{\cat{C}}(P,X_{\bu}) \ar[r]^{f_*} & \Hom_{\cat{C}}(P,Y_{\bu})
}
\]
is a fibration (resp. weak equivalence) of simplicial sets.
\item \emph{cofibration} if it has the left lifting property with respect to all trivial fibrations.
\end{itemize}
\end{definition}

\begin{remark}
Quillen's original construction only assumed finite limits and colimits. As explained in \cite{Hovey99}*{\S 1.1}, by assuming the existence of all small limits and colimits, one loses little generality and gains much convenience. 
\end{remark}

\begin{definition} \label{QuasiAlgCat}
A category is \Def{quasi-algebraic} if it is cocomplete and has a set of finitely presentable projective generators.

This notion is called a \emph{multi-sorted quasi-algebraic category} in \cite{Pedicchio04}*{Chapter VI, \S 3.1, 4.2}, where the term \emph{quasi-algebraic category} is reserved for the case where said set of generators consists of a single object. 
\end{definition}

In \cite{Quillen67}*{\S II.4, Theorem 4}, Quillen shows that quasi-algebraic categories have nice simplicial objects. In \cite{Quillen70}*{\S 2}, he proposes the word ``algebraic'' for categories as in Definition \ref{QuasiAlgCat} and then provides examples from Lawvere's work, in which the categories are assumed to be exact. We reserve the word ``algebraic'' for the more restrictive sense in Definition \ref{def:AlgCat}, for reasons which will be explained below.

Quasi-algebraic categories have excellent properties. They are locally finitely presentable (by \cite{Adamek94}*{Theorem 1.11}) and in particular complete (by \cite{Adamek94}*{Corollary 1.28}), and they are regular (by \ref{EnoughProjReg}).

\begin{remark} \label{PosetNotQuasi}
Not every locally finitely presentable category is quasi-algebraic. For example, the category $\Poset$ of partially ordered sets is locally finitely presentable, by \cite{Adamek94}*{Example 1.10 (1)}, but it is not quasi-algebraic, by \cite{Adamek94}*{Remark 3.21 (1)}. Indeed, the only projective posets are the discrete ones. The regular epimorphisms in $\Poset$ are the surjective maps $f \colon X \to Y$ such that the image $f(X)$ generates $Y$ under composition (viewing posets as categories). The poset with two comparable elements $Y = \{ y_0 < y_1 \}$ admits no regular epimorphism $P \to Y$ from a projective $P$, hence $\Poset$ does not have enough projectives.

Note that the $\cat{P}$-epic maps in $\Poset$ are those that are regular epimorphisms of underlying sets, namely the surjective maps.
\end{remark}

\begin{definition} \label{def:AlgTh}
An \Def{algebraic theory} is a small category $\cat{T}$ with finite products \cite{Adamek11}*{Definition 1.1}. A \Def{model} for the theory $\cat{T}$ is a functor $M \colon \cat{T} \to \Set$ that preserves finite products. Morphisms between models for $\cat{T}$ are natural transformations. The category $\Model(\cat{T})$ of models for $\cat{T}$ is a full subcategory of the functor category $\Set^{\cat{T}}$.

More generally, given a category $\cat{C}$, a \Def{$\cat{C}$-valued model} for $\cat{T}$ is a functor $M \colon \cat{T} \to \cat{C}$ that preserves finite products. The category $\Model(\cat{T};\cat{C})$ of $\cat{C}$-valued models for $\cat{T}$ is a full subcategory of the functor category $\cat{C}^{\cat{T}}$.

An algebraic theory $\cat{T}$ is \Def{one-sorted} if there is an object $T_1$ of $\cat{T}$ such that every object $T$ is a finite product $T \cong T_1^{n}$ for some $n \geq 0$. In that case, there is a forgetful functor \linebreak $U \colon \Model(\cat{T};\cat{C}) \to \cat{C}$, which evaluates a model $M \colon \cat{T} \to \cat{C}$ at the object $T_1$. This functor $U$ is well-defined up to natural isomorphism, namely up to the choice of the object $T_1$; cf. \cite{Borceux94v2}*{Proposition 3.3.3}.
\end{definition}

\begin{definition} \label{def:AlgCat}
A category is \Def{algebraic} if it is equivalent to the category $\Model(\cat{T})$ of models for some algebraic theory $\cat{T}$ \cite{Adamek11}*{Definition 1.2}.
\end{definition}


Every algebraic category is quasi-algebraic, with a set of finitely presentable projective generators consisting of free objects. The difference between algebraic and quasi-algebraic categories lies in exactness.

\begin{theorem} \label{AlgExact}
\cite{Adamek94}*{Corollary 3.25} A category is algebraic if and only if it is quasi-algebraic and exact.
\end{theorem}

\subsection{Varieties of algebras}

Our choice of terminology for (quasi-)algebraic categories comes from universal algebra. Let us recall the basics about varieties of algebras; more details can be found in \cite{Adamek94}*{\S 3.A, 3.B} and \cite{Adamek04qua}*{\S 1.1}.

\begin{definition}
Let $S$ be a set, called the set of \emph{sorts}. An \Def{$S$-sorted signature} is a set $\Si$, called the set of \emph{operation symbols}, together with an \Def{arity} function which assigns to each $\si \in \Si$ a family of sorts $(s_1, \ldots, s_n)$ (with $n \geq 0$) called the \emph{input sorts} of $\si$ along with a sort $s$ is called the \emph{output sort}. We denote the arity by $s_1 \x \ldots \x s_n \to s$. In the one-sorted case, i.e., when $S = \{ \ast \}$ consists of one element, the arity can be identified with the number $n$ of inputs.

A \Def{$\Si$-algebra} $A$ consists of sets $A_s$ for each sort $s \in S$ together with maps (called \emph{operations})
\[
\si_A \colon A_{s_1} \x \ldots \x A_{s_n} \to A_s
\]
for each operation symbol $\si \in \Si$ of arity $s_1 \x \ldots \x s_n \to s$. A map $f \colon A \to B$ of $\Si$-algebras consists of maps $f_s \colon A_s \to B_s$ for each sort $s \in S$ which commute with all operations. Let $\Alg \Si$ denote the category of $\Si$-algebras.

An \Def{equation} is a pair of terms $(\tau, \tau')$ of the same sort in a free $\Si$-algebra. Equations are denoted symbolically by $\tau = \tau'$. A $\Si$-algebra $A$ \Def{satisfies} the equation $\tau = \tau'$ of sort $s$ if substitution of any inputs from the respective sets $A_{s_i}$ for the formal variables in $\tau$ and $\tau'$ yields the same element in $A_s$.  

A \Def{variety} (or \emph{equational class}) of $\Si$-algebras is a full subcategory of $\Alg \Si$ of the form $\Alg (\Si,E)$, consisting of the $\Si$-algebras that satisfy all equations in some set of equations $E$.

A \emph{variety of $S$-sorted finitary algebras} is a variety of $\Si$-algebras for some $S$-sorted signature $\Si$. Here, \emph{finitary} refers to the fact that all operations have finitely many inputs. A \emph{variety of many-sorted finitary algebras}, or \emph{many-sorted finitary variety} for short, is a variety of $S$-sorted finitary algebras for some set $S$.
\end{definition}

\begin{example} \label{AbGpsVar}
Abelian groups form a one-sorted finitary variety. Let $S = \{ \ast \}$, and let $\Si = \{ \mu, e, \io \}$ be operation symbols of arity $2$, $0$, and $1$ respectively. Consider the set $E$ of four equations:
\begin{align*}
&\mu(\mu(x,y),z) = \mu(x,(\mu(y,z)) \\
&\mu(x,e) = x \\
&\mu(x, \io(x)) = e \\
&\mu(x,y) = \mu(y,x).
\end{align*}
Then abelian groups are precisely $(\Si,E)$-algebras, and we have $\Ab = \Alg (\Si,E)$.

Likewise, monoids, groups, rings, commutative rings, Lie algebras, $R$-modules, $R$-algebras, and commutative $R$-algebras (for a fixed commutative ring $R$) are one-sorted finitary varieties.
\end{example}

\begin{definition}
An \Def{implication} is a finite list of equations $\tau_i = \tau_i'$ called \emph{premises} together with an equation $\tau = \tau'$ called \emph{conclusion}. We denote an implication symbolically by:
\[
\bigwedge_i \tau_i = \tau_i' \Ra \tau = \tau'.
\]
A $\Si$-algebra $A$ \Def{satisfies} such an implication if for any inputs from the respective sets $A_s$ satisfying all premises, these inputs also satisfy the conclusion.

A \Def{quasivariety} (or \emph{implicational class}) of $\Si$-algebras is a full subcategory of $\Alg \Si$ of the form $\Alg (\Si,E,I)$, consisting of the $\Si$-algebras that satisfy a set of equations $E$ and a set of implications $I$.

\emph{Many-sorted finitary quasivarieties} are defined analogously to varieties.
\end{definition}

One has the following universal-algebraic characterization theorems. The first is due to Lawvere, at least in the one-sorted case; the second is due to Isbell.

\begin{theorem} \label{LawvereCharac}
\cite{Adamek94}*{Theorem 3.16, Remark 3.17} A category is algebraic if and only if it is equivalent to a many-sorted finitary variety.
\end{theorem}

\begin{remark}
Given a many-sorted finitary variety $\cat{C}$, the \emph{theory of $\cat{C}$} is the algebraic theory $\cat{T}_{\cat{C}} \dfn \cat{C}_{\mathrm{ff}}^{\opp}$, the opposite of the full subcategory $\cat{C}_{\mathrm{ff}}$ of $\cat{C}$ consisting of finitely generated free objects. One direction of Theorem \ref{LawvereCharac} is the equivalence $\cat{C} \cong \Model(\cat{T}_{\cat{C}})$.
\end{remark}

\begin{remark}
Abelian group objects in $\cat{C}$ can be described as $\cat{C}$-valued models of the theory $\cat{T}_{\Ab}$ of abelian groups: $\cat{C}_{\ab} \cong \Model(\cat{T}_{\Ab};\cat{C})$. Via this equivalence, the forgetful functor $U \colon \Model(\cat{T}_{\Ab};\cat{C}) \to \cat{C}$ from \ref{def:AlgTh} is the usual forgetful functor $U \colon \cat{C}_{\ab} \to \cat{C}$.
\end{remark}

\begin{theorem}[Isbell's Characterization Theorem] \label{IsbellCharac}
\cite{Adamek94}*{Theorem 3.24} A category is quasi-algebraic if and only if it is equivalent to a finitary quasivariety.
\end{theorem}

\begin{example} \label{TorsionFreeQuasi}
Consider the category $\tfAb$ of torsion-free abelian groups, viewed as a full subcategory of abelian groups. In the notation of \ref{AbGpsVar}, we have the equational presentation $\Ab = \Alg (\Si,E)$. Consider the set of implications $I$ given by:
\begin{align*}
&\mu(x,x) = e \Ra x = e \\
&\mu(x,\mu(x,x)) = e \Ra x = e \\
&\ldots
\end{align*}
Then torsion-free abelian groups are precisely those satisfying all implications in $I$, so that $\tfAb$ is the quasivariety $\Alg(\Si,E,I)$.

By \ref{IsbellCharac}, $\tfAb$ is a quasi-algebraic category, as one can check directly. The inclusion $\tfAb \inj \Ab$ has a left adjoint, which quotients out the torsion subgroup. Thus $\tfAb$ is cocomplete. Moreover, $\Z$ is a finitely presentable projective generator for $\tfAb$, as it is for $\Ab$.

Presenting $\tfAb$ as a quasivariety does not a priori exclude that it be a variety. However, by \ref{AlgExact}, $\tfAb$ is not a variety since it is not exact. For any integer $n \geq 2$, the map $n \colon \Z \to \Z$ in $\tfAb$ is a monomorphism which is not the kernel of its cokernel. Indeed, its cokernel is $\Z \to 0$, whose kernel is $1 \colon \Z \to \Z$. In other words, the equivalence relation $\{ (x,y) \in \Z \x \Z \mid x \equiv y \mod (n) \}$ on $\Z$ is not effective.
\end{example}

\begin{example} \label{ReducedQuasi}
Let $\Com$ denote the category of commutative rings, and $\redCom$ the full subcategory consisting of reduced commutative rings, i.e., those without nilpotents. Then $\Com$ is  variety $\Alg (\Si,E)$, with signature $\Si = \{ \mu, e, \io, m, u \}$ where $\mu, e, \io$ represent the addition, zero element, and negative in the underlying abelian group, and $m$ and $u$ represent the multiplication and unit element. Consider the set of implications $I$ given by:
\begin{align*}
&m(x,x) = e \Ra x = e \\
&m(x,m(x,x)) = e \Ra x = e \\
\ldots
\end{align*}
Then reduced commutative rings are precisely those satisfying all implications in $I$, so that $\redCom$ is the quasivariety $\Alg(\Si,E,I)$.

Again, one can check directly that $\redCom$ is quasi-algebraic. The inclusion $\redCom \inj \Com$ has a left adjoint, which quotients out the nilradical. Thus $\redCom$ is cocomplete. Moreover, the free commutative ring on one generator, the polynomial ring $\Z[x]$, is a finitely presentable projective generator for $\redCom$, as it is for $\Com$.

However, $\redCom$ is not a variety since it is not exact. Consider the equivalence relation $R = \{ (x,y) \in \Z \x \Z \mid x \equiv y \mod (4) \}$ on $\Z$. The coequalizer of $R$ in $\redCom$ is $\Z \surj \Z/2$, whose kernel pair is $\{ (x,y) \in \Z \x \Z \mid x \equiv y \mod (2) \}$, hence $R$ is not effective.
\end{example}

A useful generalization of algebraic theories is provided by sketches, as introduced by Ehresmann.

\begin{definition}
\cite{Adamek94}*{Definition 1.49} \cite{Borceux94v2}*{Definition 5.6.1} A \Def{limit sketch} consists of a pair $(\cat{S},\cat{L})$ where $\cat{S}$ is a small category, and $\cat{L}$ is a set of limiting cones over small diagrams in $\cat{S}$. A \Def{finite product sketch}, or \Def{FP sketch} for short, is a limit sketch where the limiting cones in $\cat{L}$ are over finite discrete diagrams in $\cat{S}$, i.e., the cones are finite products.

Note that an algebraic theory $\cat{T}$ can be viewed as an FP sketch, where we take $\cat{L}$ to be the set of all finite products in $\cat{T}$.
\end{definition}

\begin{definition}
A \Def{model} for the limit sketch $(\cat{S},\cat{L})$ is a functor $M \colon \cat{S} \to \Set$ sending the cones in $\cat{L}$ to limiting cones in $\Set$; in other words, $M$ preserves the specified limits. Morphisms between models are natural transformations. The category $\Model(\cat{S},\cat{L})$ of models for $(\cat{S},\cat{L})$ is a full subcategory of the functor category $\Set^{\cat{S}}$.
\end{definition}

\begin{definition}
A category is \Def{FP sketchable} if it is equivalent to the category $\Model(\cat{S},\cat{L})$ of models for some FP sketch $(\cat{S},\cat{L})$.
\end{definition}

\begin{theorem}
\cite{Adamek94}*{Theorem 3.16, Remark 3.17} A category is FP sketchable if and only if it is algebraic.
\end{theorem}

Therefore, FP sketches provide a more general way to describe an algebraic category, but do not provide a broader class of categories of models.

\section{Setup for Quillen (co)homology} \label{sec:SetupHQ}

In this section, we study in more detail the categorical assumptions needed in order to work with Quillen cohomology. Most importantly, we want the prolonged adjunction $Ab_X \colon s \cat{C}/X \rla s (\cat{C}/X)_{\ab} \colon U_X$ to be a Quillen pair.

\subsection{Prolonged adjunctions as Quillen pairs} \label{sec:QuilPairs}

Recall the following useful fact, giving sufficient conditions for the regular epimorphisms to be determined by the projectives.

\begin{proposition}
\cite{Quillen67}*{\S II.4, Proposition 2} If a category $\cat{C}$ has finite limits and enough projectives, then every $\cat{P}$-epic map in $\cat{C}$ is a regular epimorphism. In other words, a map $f \colon X \to Y$ is a regular epimorphism if and only if the map:
\[
f_* \colon \Hom_{\cat{C}}(P,X) \to \Hom_{\cat{C}}(P,Y)
\]
is a surjection for every projective $P$. 
\end{proposition}

\begin{corollary} \label{EnoughProjReg}
If $\cat{C}$ has finite limits and enough projectives, then regular epimorphisms in $\cat{C}$ are closed under pullbacks.

In particular, if moreover $\cat{C}$ has coequalizers of kernel pairs, then $\cat{C}$ is regular.
\end{corollary}


\begin{proposition}
Assume we have an adjunction $F \colon \cat{C} \rla \cat{D} \colon G$. Then $G$ sends regular epimorphisms to $\cat{P}$-epic maps if and only if $F$ preserves projectives.

In particular, if $\cat{C}$ has finite limits and enough projectives, then the condition is equivalent to $G$ preserving regular epimorphisms.
\end{proposition}

\begin{proof}
Given a map $f \colon d \to d'$ in $\cat{D}$ and an object $P$ in $\cat{C}$, consider the commutative diagram of sets:
\[
\xymatrix{
\Hom_{\cat{D}}(FP,d) \ar[d]_{\cong} \ar[r]^{f_*} & \Hom_{\cat{D}}(FP,d') \ar[d]^{\cong} \\
\Hom_{\cat{C}}(P,Gd) \ar[r]^{(Gf)_*} & \Hom_{\cat{C}}(P,Gd'). \\
}
\]
in which the top map is surjective if and only if the bottom map is surjective. Surjectivity for every projective $P$ in $\cat{C}$ and every regular epimorphism $f \colon d \to d'$ in $\cat{D}$ is equivalent to $FP$ being projective in $\cat{D}$ for every projective $P$ in $\cat{C}$, and also equivalent to $Gf \colon Gd \to Gd'$ being $\cat{P}$-epic in $\cat{C}$ for every regular epimorphism $f \colon d \to d'$ in $\cat{D}$.
\end{proof}

\begin{lemma} \label{Cofibrant}
Let $\cat{C}$ be a category with nice simplicial objects, as in Definition \ref{NiceSplObj}. Then every cofibrant object of $s\cat{C}$ is degreewise projective.
\end{lemma}

\begin{proof}
Let $\mathrm{ev}_n \colon s\cat{C} \to \cat{C}$ be the functor evaluating at $\ord{n}$, and let $r_n \colon \cat{C} \to s\cat{C}$ be its right adjoint, which can be described as a right Kan extension. 
Let $C_{\bu}$ be a cofibrant object in $s\cat{C}$. We want to show that $C_n$ is projective. The lifting problem
\[
\xymatrix{
& X \ar@{->>}[d]^{f} \\
C_n \ar@{-->}[ur] \ar[r] & Y \\
}
\]
for a regular epimorphism $f \colon X \surj Y$ in $\cat{C}$ is equivalent, by adjunction, to the lifting problem
\[
\xymatrix{
& r_n X \ar[d]^{r_n f} \\
C_{\bu} \ar@{-->}[ur] \ar[r] & r_n Y \\
}
\]
in $s\cat{C}$. Since $C_{\bu}$ is cofibrant, it suffices to show that $r_n f$ is a trivial fibration in $s\cat{C}$ to guarantee the existence of such a lift.

To show that $r_n f$ is a trivial fibration, let $P$ be a projective object of $\cat{C}$, and consider the map of simplicial sets
\[
\xymatrix{
\Hom_{\cat{C}}(P,r_n X) \ar[r]^{(r_n f)_*} & \Hom_{\cat{C}}(P,r_n Y).
}
\]
Recall the isomorphism of simplicial sets from \cite{Goerss99}*{Proof of Theorem II.2.5}
\[
\Hom_{\cat{C}}(P,Z_{\bu}) \cong \HOM_{s\cat{C}}(P_c,Z_{\bu})
\]
where $\HOM_{s\cat{C}}(V_{\bu},Z_{\bu})$ is the usual simplicial mapping space in $s\cat{C}$, whose degree $n$ object is $\HOM_{s\cat{C}}(V_{\bu},Z_{\bu})_n = \Hom_{s\cat{C}}(V_{\bu} \ot \De^n,Z_{\bu})$. To show that the map $(r_n f)_*$ is a trivial fibration, we test it against an arbitrary cofibration of simplicial sets:
\[
\xymatrix{
A_{\bu} \ar[d]_{i_{\bu}} \ar[r] & \Hom_{\cat{C}}(P,r_n X) \ar[d]^{(r_n f)_*} \\ 
B_{\bu} \ar@{-->}[ur] \ar[r] & \Hom_{\cat{C}}(P,r_n Y). \\
}
\]
By adjunction, this lifting problem in $s\Set$ is equivalent to the lifting problem in $s\cat{C}$:
\[
\xymatrix{
P_c \ot A_{\bu} \ar[d]_{P_c \ot i_{\bu}} \ar[r] & r_n X \ar[d]^{r_n f} \\ 
P_c \ot B_{\bu} \ar@{-->}[ur] \ar[r] & r_n Y \\
}
\]
which, again by adjunction, is equivalent to the lifting problem in $\cat{C}$:
\begin{equation} \label{eq:LiftingProb}
\xymatrix{
(P_c \ot A_{\bu})_n \ar[d]_{(P_c \ot i_{\bu})_n} \ar[r] & X \ar[d]^{f} \\ 
(P_c \ot B_{\bu})_n \ar@{-->}[ur] \ar[r] & Y. \\
}
\end{equation}
Using the explicit description of the tensoring $\ot \colon s\cat{C} \x s\Set \to s\cat{C}$, the map $(P_c \ot i_{\bu})_n$ on the left can be written as:
\begin{equation} \label{eq:SummandIncl}
\coprod_{a \in A_n} P \to \coprod_{b \in B_n} P
\end{equation}
induced by the map of sets $i_n \colon A_n \to B_n$. Since $i_{\bu} \colon A_{\bu} \to B_{\bu}$ is a cofibration of simplicial sets, $i_n \colon A_n \to B_n$ is injective and the map \eqref{eq:SummandIncl} is the inclusion of the corresponding summands. Since $P$ is projective in $\cat{C}$, the lifting problem \eqref{eq:LiftingProb} has a solution.
\end{proof}

\begin{lemma} \label{SplModelStruc}
Let $\cat{C}$ be a category with nice simplicial objects.
\begin{enumerate}
\item \label{LevelwiseRegEpi} Every trivial fibration $f_{\bu} \colon X_{\bu} \to Y_{\bu}$ in $s\cat{C}$ is degreewise $\cat{P}$-epic. In particular, if $\cat{C}$ has enough projectives, then $f_{\bu}$ is a degreewise regular epimorphism.

\item \label{LevelZero} If $f \colon X \to Y$ is a $\cat{P}$-epic map in $\cat{C}$, then there is a trivial fibration $f_{\bu} \colon X_{\bu} \to Y_c$ in $s\cat{C}$ whose degree $0$ part is $f_0 = f$.

\item \label{EpicFromProj} Every object $X$ of $\cat{C}$ admits a $\cat{P}$-epic map $P \to X$ from a projective $P$.

\item \label{EnoughProj} $\cat{C}$ has enough projectives if and only if all trivial fibrations in $s\cat{C}$ are degreewise regular epimorphisms. 
\end{enumerate}
\end{lemma}

\begin{proof}
1. Let $P$ be a projective object in $\cat{C}$. By definition of trivial fibrations in $s\cat{C}$, the map
\[
\xymatrix{
\Hom_{\cat{C}}(P,X_{\bu}) \ar[r]^{f_*} & \Hom_{\cat{C}}(P,Y_{\bu})
}
\]
is a trivial fibration of simplicial sets, in particular a degreewise surjection. Therefore each $f_n$ is $\cat{P}$-epic.

2. Since $f \colon X \to Y$ coequalizes the two projections of the kernel pair $X \x_Y X \rra X$ (though $f$ need not be the coequalizer), one can form the augmented simplicial object:
\[
\xymatrix{
\ldots & X \x_Y X \x_Y X \ar@<1.2ex>[r] \ar[r] \ar@<-1.2ex>[r] & X \x_Y X \ar@<0.6ex>[r] \ar@<-0.6ex>[r] & X \ar[r]^f & Y \\
}
\]
which can be viewed as a map $f_{\bu} \colon X_{\bu} \to Y_c$ in $s\cat{C}$. Here, $X_{\bu}$ is the simplicial object with $X_n = X \x_Y \ldots \x_Y X$ ($n+1$ factors), where faces are given by projections and degeneracies are given by diagonals.

Let us show that $f_{\bu}$ is a trivial fibration in $s\cat{C}$. Without loss of generality, we may assume $\cat{C} = \Set$, since for every projective $P$ of $\cat{C}$, the functor $\Hom_{\cat{C}}(P,-) \colon \cat{C} \to \Set$ preserves limits and sends $\cat{P}$-epic maps to surjections. Since $Y_c$ is constant, the assertion is equivalent to this: the preimage by $f \colon X_{\bu} \to Y_c$ over each point $y \in Y$ is a contractible Kan complex. Since $f_0 = f \colon X \to Y$ is surjective, we may assume that $Y$ is a point and $X$ is non-empty.

Now $X_{\bu} = N(EX)$ is the nerve of the contractible groupoid (or indiscrete category) $EX$ on the set $X$, where $EX$ has as objects the elements of $X$, and exactly one morphism from $x$ to $x'$ for any $x,x' \in X$. Hence $X_{\bu}$ is a contractible Kan complex.

3. Let $C_{\bu} \to X_c$ be a trivial fibration in $s\cat{C}$ from a cofibrant object to $X$ viewed as a constant simplicial object. Then $C_0$ is projective, by \ref{Cofibrant}, and the map $C_0 \to X$ is $\cat{P}$-epic, by part \ref{LevelwiseRegEpi}.

4. If $\cat{C}$ has enough projectives, then every $\cat{P}$-epic map in $\cat{C}$ is a regular epimorphism, and by part \ref{LevelwiseRegEpi}, trivial fibrations in $s\cat{C}$ are degreewise regular epimorphisms. Conversely, if trivial fibrations in $s\cat{C}$ are degreewise regular epimorphisms, then the map $C_0 \to X$ constructed in part \ref{EpicFromProj} is a regular epimorphism from a projective.
\end{proof}

\begin{remark}
By \hyperref[EpicFromProj]{\ref*{SplModelStruc} (\ref*{EpicFromProj})}, the class $\cat{P}$ of regular projectives and the $\cat{P}$-epic maps then form a \emph{projective class} in the sense of \cite{Christensen02}*{Definition 1.1}; see also \cite{Christensen02}*{\S 6.2}.
\end{remark}

\begin{remark}
The converse of part \ref{LevelwiseRegEpi} is false, even for $\cat{C} = \Set$. For example, recall that a map of constant simplicial sets is always a fibration, and is a weak equivalence if and only if it is bijective. Now take a surjective map of sets $f \colon X \surj Y$ which is not bijective, viewed as a map of constant simplicial sets $f_c \colon X_c \to Y_c$. Then $f_c$ is a degreewise regular epimorphism but not a trivial fibration. 
\end{remark}

\begin{remark}
Having nice simplicial objects does not guarantee having enough projectives. For example, consider the category $\cat{C} = \Poset$, which is locally finitely presentable. Recall from \ref{PosetNotQuasi} that the projectives in $\Poset$ are precisely the discrete posets. Therefore, the notions in Definition \ref{NiceSplObj} are that a map $f_{\bu} \colon X_{\bu} \to Y_{\bu}$ of simplicial posets is a fibration (resp. weak equivalence) if the map of underlying simplicial sets $Uf_{\bu} \colon UX_{\bu} \to UY_{\bu}$ is a fibration (resp. weak equivalence) of simplicial sets. By \cite{Goerss99}*{Corollary II.5.6}, these notions do define a (closed) model structure on $s\Poset$, in fact a simplicial model structure.

By \hyperref[EnoughProj]{\ref*{SplModelStruc} (\ref*{EnoughProj})}, $s\Poset$ has some trivial fibrations which are not degreewise regular epimorphisms. Let us describe an explicit example thereof. A map of constant simplicial posets is always a fibration, and is a weak equivalence if and only if it is bijective. In particular, consider the discrete two-element poset $X = \{ x_0, x_1 \}$, the non-discrete two-element poset $Y = \{ y_0 < y_1 \}$, and the map $f \colon X \to Y$ defined by $f(x_i) = y_i$ for $i=0,1$. Then the map of constant simplicial posets $f_c \colon X_c \to Y_c$ is a trivial fibration in $s\Poset$, but it is not a regular epimorphism in any degree.
\end{remark}

\begin{proposition} \label{ProlongQuilPair}
Assume $\cat{C}$ and $\cat{D}$ have nice simplicial objects. Assume we have an adjunction $F \colon \cat{C} \rla \cat{D} \colon G$, and hence a prolonged adjunction:
\[ 
\xymatrix{
s\cat{C} \ar@<0.6ex>[r]^-{F} & s\cat{D} \ar@<0.6ex>[l]^-{G}
}
\]
between model categories. Then this prolonged adjunction is a Quillen pair if and only if $F$ preserves projectives, or equivalently, if $G$ sends regular epimorphisms to $\cat{P}$-epic maps.
\end{proposition}

\begin{proof}
($\Ra$) Take a regular epimorphism $f \colon X \to Y$ in $\cat{D}$ and consider a trivial fibration $f_{\bu} \colon X_{\bu} \to Y_c$ in $s\cat{D}$ satisfying $f_0 = f$, as in \hyperref[LevelZero]{\ref*{SplModelStruc} (\ref*{LevelZero})}. Since $G$ prolongs to a right Quillen functor, $G f_{\bu}$ is a trivial fibration in $s\cat{C}$, and hence degreewise $\cat{P}$-epic. In particular, $Gf = Gf_0$ is $\cat{P}$-epic.

($\Leftarrow$) We show a slightly stronger statement: $G$ preserves fibrations and weak equivalences. Take a fibration (resp. weak equivalence) $f \colon X_{\bu} \to Y_{\bu}$ in $s\cat{D}$ and a projective $P$ in $\cat{C}$, and consider:
\[
\xymatrix{
\Hom_{\cat{C}}(P,GX_{\bu}) \ar[d]_{\cong} \ar[r]^{(Gf)_*} & \Hom_{\cat{C}}(P,GY_{\bu}) \ar[d]^{\cong} \\
\Hom_{\cat{D}}(FP,X_{\bu}) \ar[r]^{f_*} & \Hom_{\cat{D}}(FP,Y_{\bu}). \\
}
\]
By assumption, $FP$ is projective in $\cat{D}$, hence the bottom and top maps are fibrations (resp. weak equivalences) of simplicial sets. Thus $Gf \colon GX_{\bu} \to GY_{\bu}$ is a fibration (resp. weak equivalence) in $s\cat{C}$.
\end{proof}

\begin{remark}
We have seen that a prolonged right Quillen functor in \ref{ProlongQuilPair} is particularly strong: it preserves fibrations and \emph{all} weak equivalences, not just between fibrant objects. However, the prolonged left Quillen functor does not enjoy this additional property in general, i.e., it need not preserve all weak equivalences, only those between cofibrant objects.
\end{remark}

\begin{example}
Let $R$ be a commutative ring and consider the functor $R \ot -$ from abelian groups to $R$-modules. It preserves projectives (since it sends a free abelian group to a free $R$-module), but the prolonged left Quillen functor does not preserve all weak equivalences if $R$ is not flat over $\Z$.
\end{example}

\subsection{Slice categories}

Proposition \ref{ProlongQuilPair} gives a simple criterion for when a prolonged adjunction is a Quillen pair. We want to know if the induced adjunction on slice categories is also a Quillen pair. Let us first describe regular epimorphisms and projectives in the slice category.

\begin{proposition} \label{EpiSlice}
If $f \colon Y \to Z$ is a regular epimorphism in $\cat{C}$, then:
\[
\xymatrix{
Y \ar[dr] \ar[r]^f & Z \ar[d] \\
& X \\
}
\]
is a regular epimorphism in $\cat{C}/X$. The converse also holds if $\cat{C}$ has coequalizers.
\end{proposition}

\begin{proof}
See \cite{Barr02}*{Chapter 1, Proposition 8.12}. It follows from the fact that the ``source'' forgetful functor $\cat{C}/X \to \cat{C}$ creates colimits.
\end{proof}

\begin{proposition} \label{ProjSlice}
\begin{enumerate}
\item If $P$ is projective in $\cat{C}$, then $p \colon P \to X$ is projective in $\cat{C}/X$.
\item The converse also holds if $\cat{C}$ has enough projectives.
\end{enumerate}
\end{proposition}

\begin{proof}
1. Start with a regular epimorphism:
\[ 
\xymatrix{
Y \ar[dr]_y \ar[r]^f & Z \ar[d]^z \\
& X \\
}
\]
in $\cat{C}/X$, which means $f \colon Y \to Z$ is a regular epimorphism in $\cat{C}$, by \ref{EpiSlice}. We want to know if the map:
\[ 
f_* \colon \Hom_{\cat{C}/X} (P \ral{p} X, Y \ral{y} X) \to \Hom_{\cat{C}/X} (P \ral{p} X, Z \ral{z} X)
\]
is surjective. Let $\al$ be a map in the right-hand side which we are trying to reach and consider the diagram:
\[ 
\xymatrix{
& Y \ar[ddr]_y \ar[rr]^f & & Z \ar[ddl]^z \\
P \ar@{-->}[ur]^{\tild{\al}} \ar[urrr]^{\al} \ar[drr]_p & & & \\
& & X & \\
}
\]
Since $P$ is projective in $\cat{C}$, there is a lift $\tild{\al}$ in the top triangle, meaning $f \tild{\al} = \al$. If $\tild{\al}$ is in fact a map in $\Hom_{\cat{C}/X} (P \ral{p} X, Y \ral{y} X)$, then it will be our desired lift. So it suffices to check that the triangle on the left commutes: $y \tild{\al} = z f \tild{\al} = z \al = p$.

2. Let $E \ral{e} X$ be projective in $\cat{C}/X$. Since $\cat{C}$ has enough projectives, pick a regular epimorphism $\pi \colon P \to E$ from a projective $P$. Consider the diagram:
\[
\xymatrix{
& P \ar[ddr]_{e \pi} \ar[rr]^{\pi} & & E \ar[ddl]^e \\
E \ar@{-->}[ur]^{s} \ar[urrr]^{\id} \ar[drr]_e & & & \\
& & X & \\
}
\]
where there exists a lift $s$ since $E \ral{e} X$ is projective in $\cat{C}/X$. The relation $\pi s = \id_E$ exhibits $E$ as a retract of a projective in $\cat{C}$, hence itself projective.
\end{proof}

Now we can describe the standard Quillen model structure on $s(\cat{C}/X) \cong s\cat{C} / X$. A map:
\begin{equation} \label{MapSlice}
\xymatrix{
Y_{\bu} \ar[dr]_y \ar[r]^f & Z_{\bu} \ar[d]^z \\
& X \\
}
\end{equation}
is a fibration (resp. weak equivalence) in $s(\cat{C}/X)$ if and only if the map:
\[
\xymatrix{
\Hom_{\cat{C}/X} (P \ral{p} X, Y_{\bu} \ral{y} X) \ar[r]^{f_*} & \Hom_{\cat{C}/X} (P \ral{p} X, Z_{\bu} \ral{z} X) \\
}
\]
is a fibration (resp. weak equivalence) of simplicial sets for all projective $P \ral{p} X$ in $\cat{C}/X$. By Proposition \ref{ProjSlice}, we can rephrase the latter as: for all projective $P$ in $\cat{C}$ and map $p \in \Hom_{\cat{C}}(P,X)$.

However, in the framework of Quillen (co)homology, we decided to work with the ``slice'' model structure on $s\cat{C} / X$, where the map \eqref{MapSlice} is a fibration (resp. weak equivalence) if and only if the map $f_* \colon \Hom_{\cat{C}} (P, Y_{\bu}) \to \Hom_{\cat{C}} (P, Z_{\bu})$ is a fibration (resp. weak equivalence) of simplicial sets for all projective $P$ in $\cat{C}$. In fact, let us check that the two model structures agree.

\begin{proposition}
There is a natural isomorphism of simplicial sets:
\[
\xymatrix{
\coprod_{p \in \Hom_{\cat{C}}(P,X)} \Hom_{\cat{C}/X} (P \ral{p} X, Y_{\bu} \ral{y} X) \ar[r]^-{\cong} & \Hom_{\cat{C}} (P, Y_{\bu}) . \\
}
\]
\end{proposition}

\begin{proof}
For a fixed $y \colon Y \to X$, a map $g \colon P \to Y$ is the same as the data of the commutative diagram:
\[
\xymatrix{
P \ar[dr]_p \ar[r]^g & Y \ar[d]^y \\
& X \\
}
\]
and thus we can partition all maps $g \colon P \to Y$ according to their composite $p = yg \colon P \to X$. More precisely, we take the map
\[
\coprod_{p \in \Hom_{\cat{C}}(P,X)} \Hom_{\cat{C}/X} (P \ral{p} X, Y \ral{y} X) \to \Hom_{\cat{C}} (P, Y)
\]
which is readily seen to be surjective and injective, i.e., an isomorphism of sets. Moreover, it is natural in $y \colon Y \to X$, i.e., the two sides define two naturally isomorphic functors from $\cat{C}/X$ to $\Set$. By naturality, it prolongs to a natural isomorphism of simplicial sets. Since colimits of simplicial objects are computed degreewise, the simplicial set whose $n^{\text{th}}$ degree is
\[
\left\{ \coprod_{p \in \Hom_{\cat{C}}(P,X)} \Hom_{\cat{C}/X} (P \ral{p} X, Y_n \ral{y_n} X) \right\}_n
\]
equals the left-hand side in the statement.
\end{proof}

\begin{proposition} \label{SameModelStruc}
The standard model structures on $s(\cat{C}/X)$ and $s\cat{C} / X$ are the same.
\end{proposition}

\begin{proof}
The top row in the diagram:
\[
\xymatrix{
\Hom_{\cat{C}} (P, Y_{\bu}) \ar[r]^{f_*} & \Hom_{\cat{C}} (P, Z_{\bu}) \\
\coprod_p \Hom_{\cat{C}/X} (P \ral{p} X, Y_{\bu} \ral{y} X) \ar@{=}[u] \ar[r]^{f_*} & \coprod_p \Hom_{\cat{C}/X} (P \ral{p} X, Z_{\bu} \ral{z} X) \ar@{=}[u] \\
}
\]
is a fibration (resp. weak equivalence) of simplicial sets if and only if each summand is so. This means that $f$ is a fibration (resp. weak equivalence) in $s\cat{C} / X$ if and only if it is so in $s(\cat{C}/X)$. Moreover, the model structures are closed, i.e., cofibrations are determined by fibrations and weak equivalences (as having the left lifting property with respect to trivial fibrations). Therefore, the two model structures agree.
\end{proof}

\subsection{Regularity of abelian group objects} \label{sec:Regularity}

In this section, we study the properties of the category $\cat{C}_{\ab}$ of abelian group objects in a category $\cat{C}$ and the forgetful functor $U \colon \cat{C}_{\ab} \to \cat{C}$.

It is convenient to work with regular categories, so we would like to know if $\cat{C}_{\ab}$ is regular whenever $\cat{C}$ is. The main feature of regular categories is that any map can be factored as a regular epimorphism followed by monomorphism; we call this the \emph{regular-epi--mono factorization}, which is unique up to isomorphism \cite{Borceux94v2}*{Theorem 2.1.3}. Isomorphisms are precisely maps that are both a regular epimorphism and a monomorphism. We will check that all three classes of maps are preserved and reflected by $U$.

First, recall that $U$ is faithful, creates limits, and reflects isomorphisms. Indeed, if $Uf$ is an isomorphism in $\cat{C}$, then $(Uf)^{-1}$ lifts to a map in $\cat{C}_{\ab}$.

\begin{proposition} \label{PreserveMono}
Assume $\cat{C}$ has kernel pairs. Then $U$ preserves monomorphisms.
\end{proposition}

\begin{proof}
In a category with kernel pairs, a map $f \colon X \to Y$ is a monomorphism if and only if the two projections $X \x_Y X \rra X$ from its kernel pair are equal. Thus, any functor between categories with kernel pairs which preserves kernel pairs also preserves monomorphisms.
\end{proof}

In \cite{Barr02}*{Chapter 6, Proposition 1.7}, Barr shows the following.

\begin{proposition} \label{LiftFactorization}
Assume $\cat{C}$ is regular. Then $U \colon \cat{C}_{\ab} \to \cat{C}$ lifts the regular-epi--mono factorization in $\cat{C}$. In other words, if $f \colon X \to Y$ is a map in $\cat{C}_{\ab}$ and $UX \surj Z' \inj UY$ is a regular-epi--mono factorization of the underlying map $Uf$, then we can lift it (uniquely) to a factorization $X \to Z \to Y$ in $\cat{C}_{\ab}$. 
\end{proposition}


\begin{corollary} \label{PreserveRegEpi}
If $\cat{C}$ is regular, then $U \colon \cat{C}_{\ab} \to \cat{C}$ preserves regular epimorphisms.
\end{corollary}


In addition, we would like to know if $U$ reflects regular epimorphisms.

\begin{proposition} \label{CreateCoeq}
If $\cat{C}$ is regular, then $\cat{C}_{\ab}$ has coequalizers of kernel pairs, created by $U$.
\end{proposition}

\begin{proof}
Let $f \colon X \to Y$ be any map in $\cat{C}_{\ab}$ and take its kernel pair $X \x_Y X \rra X$. Since $U$ preserves limits, the underlying diagram is still a kernel pair, and we can take its coequalizer:
\[
\xymatrix{
UX \x_{UY} UX \ar@<0.6ex>[r]^-{pr_1} \ar@<-0.6ex>[r]_-{pr_2} & UX \ar[r]^{Uf} \ar@{->>}[dr] & UY \\
& & C. \ar[u]_h \\
}
\]
Since $\cat{C}$ is regular, the map $h \colon C \to Y$ is a monomorphism, by \cite{Barr02}*{Chapter 1, Proposition 8.10}. By \ref{LiftFactorization}, there is a unique lift $X \to \tild{C} \to Y$ of that regular-epi--mono factorization. One can check that $X \to \tild{C}$ is the desired coequalizer in $\cat{C}_{\ab}$ of the kernel pair of $f$.
\end{proof}

\begin{proposition} \label{ReflectRegEpi}
If $\cat{C}$ is regular, then $U \colon \cat{C}_{\ab} \to \cat{C}$ reflects regular epimorphisms.
\end{proposition}

\begin{proof}
Let $f \colon X \to Y$ be a map in $\cat{C}_{\ab}$ such that $Uf$ is a regular epimorphism in $\cat{C}$. We want to show that $f$ is a regular epimorphism. Since $U$ creates limits, the kernel pair of $f$ is the unique lift of the kernel pair $UX \x_{UY} UX \rra UX$ of $Uf$, and the latter has a coequalizer, namely $Uf \colon UX \to UY$. Since $U$ creates coequalizers of kernel pairs, there is a unique cocone lifting $Uf \colon UX \to UY$ and it is a coequalizer of $X \x_Y X \rra X$. But $f \colon X \to Y$ is such a lift, hence $f$ is a regular epimorphism.
\end{proof}

\begin{corollary}
The lifted factorization of \ref{LiftFactorization} is a regular-epi--mono factorization in $\cat{C}_{\ab}$.
\end{corollary}

\begin{corollary} \label{AbReg}
If $\cat{C}$ is regular, then $\cat{C}_{\ab}$ is regular.
\end{corollary}

\begin{proof}
$\cat{C}_{\ab}$ has kernel pairs (and any limits that $\cat{C}$ has) and coequalizers of kernel pairs. It remains to check that the pullback of a regular epimorphism is a regular epimorphism:
\[
\xymatrix{
P \ar[d]_{f^*e} \ar[r] & X \ar@{->>}[d]^e \\
W \ar[r]_f & Y. \\
}
\]
Since $U$ preserves regular epimorphisms, $Ue$ is a regular epimorphism. Since pullbacks are computed in $\cat{C}$, we have $U(f^*e) = (Uf)^*(Ue)$, which is a regular epimorphism since $\cat{C}$ is regular. Since $U$ reflects regular epimorphisms, $f^*e$ itself is a regular epimorphism in $\cat{C}_{\ab}$.
\end{proof}

For the record, let us extract a more general statement from the arguments above.

\begin{proposition}
Let $\cat{T}$ be a one-sorted algebraic theory. Then all the statements in Subsection \ref{sec:Regularity} about $U \colon \cat{C}_{\ab} \to \cat{C}$ also apply to the forgetful functor $U \colon \Model(\cat{T};\cat{C}) \to \cat{C}$. In particular, if $\cat{C}$ is regular, then $\Model(\cat{T};\cat{C})$ is regular.
\end{proposition}

\begin{proof}
By the argument of \cite{Adamek94}*{Remark 3.17} or \cite{Borceux94v2}*{\S 3.3}, there is a one-sorted finitary variety $\cat{V}$ and an equivalence of categories $\Model(\cat{T};\cat{C}) \cong \Model(\cat{T}_{\cat{V}};\cat{C})$. In \cite{Barr02}*{Chapter 6, Proposition 1.7}, Barr notes that the proofs of \ref{LiftFactorization} and \ref{PreserveRegEpi} hold more generally for $\cat{C}$-valued models of finitary equational theories, i.e., for the forgetful functor $U \colon \Model(\cat{T}_{\cat{V}};\cat{C}) \to \cat{C}$ as above. One readily checks that the remaining proofs also hold in that context.
\end{proof}

\subsection{Abelianizations and pushforwards}

\begin{proposition} \label{AbExists}
Let $\cat{C}$ be a locally $\ka$-presentable category for some regular cardinal $\ka$. Then the following holds.
\begin{enumerate}
\item \label{CreateFiltColim} $U \colon \cat{C}_{\ab} \to \cat{C}$ creates $\ka$-filtered colimits. In particular, $\cat{C}_{\ab}$ has $\ka$-filtered colimits and $U$ preserves them.
\item \label{AbCLocPres} $\cat{C}_{\ab}$ is locally $\ka$-presentable.
\item \label{AbExistsPart} $U \colon \cat{C}_{\ab} \to \cat{C}$ has a left adjoint.
 \end{enumerate}
\end{proposition}

\begin{proof}
1. This is similar to the proof that $U$ creates limits. Let $J$ be a $\ka$-filtered category and $F \colon J \to \cat{C}_{\ab}$ a diagram whose underlying diagram $UF \colon J \to \cat{C}$ admits a colimit. Then there is a unique lift of the colimiting cocone in $\cat{C}$ to a cocone in $\cat{C}_{\ab}$. Indeed, there is at most one way to endow $\colim_J UF$ with structure maps, since they are prescribed on each summand, as illustrated in the diagram:
\[
\xymatrix{
**[l] \colim_J (UF \x UF) \cong \colim_J UF \x \colim_J UF \ar[r] & \colim_J UF \\
UF(j) \x UF(j) \ar[u] \ar[r]_-{\mu} & UF(j) \ar[u] \\
}
\]
where $j$ is any object of $J$. Applying $\colim_J$ to the structure maps of $UF$ produces those structure maps for $\colim_J UF$. The result is the colimit of $F$ in $\cat{C}_{\ab}$.

We used the fact that $\ka$-filtered colimits commute with finite limits in $\cat{C}$, since $\cat{C}$ is locally $\ka$-presentable \cite{Borceux94v2}*{Corollary 5.2.8} \cite{Adamek94}*{Proposition 1.59}.

2. Recall the equivalence $\cat{C}_{\ab} \cong \Model(\cat{T}_{\Ab};\cat{C})$. Since the diagrams in the limit sketch $\cat{T}_{\Ab}$ are all finite, the result follows from \cite{Adamek94}*{Proposition 1.53}.

3. $\cat{C}$ is locally $\ka$-presentable and so is $\cat{C}_{\ab}$, by part \ref{AbCLocPres}. Morever, $U \colon \cat{C}_{\ab} \to \cat{C}$ preserves all (small) limits as well as $\ka$-filtered colimits, by part \ref{CreateFiltColim}. By the adjoint functor theorem for locally presentable categories \cite{Adamek94}*{Theorem 1.66}, $U$ has a left adjoint.  
\end{proof}

Note that the statement \hyperref[AbExistsPart]{\ref*{AbExists} (\ref*{AbExistsPart})} can be found in \cite{Barr02}*{Chapter 6 \S 1.5}. Now recall the following useful fact.

\begin{proposition} \label{LocPresSlice}
\cite{Adamek94}*{Proposition 1.57} Let $\cat{C}$ be a locally $\ka$-presentable category for some regular cardinal $\ka$, and let $X$ be an object of $\cat{C}$. Then the slice category $\cat{C}/X$ is locally $\ka$-presentable.
\end{proposition}

\begin{proposition} \label{AbExist}
Let $\cat{C}$ be a locally presentable category. Then the following holds.
\begin{enumerate}
\item \label{AbExistPart} $\cat{C}$ has all abelianizations: For every object $X$ of $\cat{C}$, the forgetful functor \linebreak $U_X \colon (\cat{C}/X)_{\ab} \to \cat{C}/X$ has a left adjoint $Ab_X \colon \cat{C}/X \to (\cat{C}/X)_{\ab}$.
\item $\cat{C}$ has all pushforwards: For every map $f \colon X \to Y$ in $\cat{C}$, the pullback functor $f^* \colon (\cat{C}/Y)_{\ab} \to (\cat{C}/X)_{\ab}$ has a left adjoint $f_{\push} \colon (\cat{C}/X)_{\ab} \to (\cat{C}/Y)_{\ab}$.
\end{enumerate}
\end{proposition}

\begin{proof}
Let $\ka$ be a regular cardinal such that $\cat{C}$ is locally $\ka$-presentable.
 
1. By \ref{LocPresSlice}, $\cat{C}/X$ is locally $\ka$-presentable. By \ref{AbExists}, $U_X \colon (\cat{C}/X)_{\ab} \to \cat{C}/X$ has a left adjoint.

2. Consider the diagram:
\[
\xymatrix{
(\cat{C}/Y)_{\ab} \ar@<0.6ex>[d]^{U_Y} \ar@<0.6ex>[r]^{f^*} & (\cat{C}/X)_{\ab} \ar@<0.6ex>[d]^{U_X} \\
\cat{C}/Y \ar@<0.6ex>[u]^{Ab_Y} \ar@<0.6ex>[r]^{f^*} & \cat{C}/X \ar@<0.6ex>[l]^{f_!} \ar@<0.6ex>[u]^{Ab_X} \\
}
\]
where the abelianizations exist, by part \ref{AbExistPart}. The right adjoints commute.

Since $(\cat{C}/Y)_{\ab}$ and $(\cat{C}/X)_{\ab}$ are locally $\ka$-presentable, it suffices to show that $f^* \colon (\cat{C}/Y)_{\ab} \to (\cat{C}/X)_{\ab}$ preserves limits and $\ka$-filtered colimits to guarantee the existence of a left adjoint (by the adjoint functor theorem for locally presentable categories). Since $U_X \colon (\cat{C}/X)_{\ab} \to \cat{C}/X$ creates limits and $\ka$-filtered colimits, by \hyperref[CreateFiltColim]{\ref*{AbExists} (\ref*{CreateFiltColim})}, it suffices to show that $U_X f^*$ preserves limits and $\ka$-filtered colimits.

Now $U_X f^* = f^* U_Y$ is a composite of right adjoints and thus preserves limits. Moreover, $U_Y$ preserves $\ka$-filtered colimits, hence it suffices to show that $f^* \colon \cat{C}/Y \to \cat{C}/X$ preserves $\ka$-filtered colimits. Since the forgetful functor $\cat{C}/X \to \cat{C}$ creates colimits, it suffices to show that the composite $\cat{C}/Y \ral{f^*} \cat{C}/X \to \cat{C}$ preserves $\ka$-filtered colimits. This holds, since $\ka$-filtered colimits and finite limits commute in $\cat{C}$, as $\cat{C}$ is locally $\ka$-presentable.
\end{proof}

We will use the following basic facts.

\begin{lemma} \label{Additive}
\begin{enumerate}
\item In a preadditive category $\cat{A}$, finite products are canonically biproducts; likewise, finite coproducts are canonically biproducts. In particular, $\cat{A}$ is additive if and only if it has finite products (or equivalently, it has finite coproducts).
\item A functor $F \colon \cat{A} \to \cat{B}$ between preadditive categories with finite powers is additive if and only if it preserves finite powers (or equivalently, it preserves finite copowers). If moreover $\cat{A}$ is additive, then $F$ is additive if and only if it preserves finite products (or equivalently, it preserves finite coproducts).
\item If a category $\cat{C}$ has finite powers, then the category $\cat{C}_{\ab}$ is preadditive. If moreover $\cat{C}$ has finite products, then $\cat{C}_{\ab}$ is additive.
\item \label{InducedFctAb} Let $\cat{C}$ and $\cat{D}$ be categories with finite powers, and let $F \colon \cat{C} \to \cat{D}$ be a functor which preserves finite powers. Then $F$ naturally induces an additive functor $F \colon \cat{C}_{\ab} \to \cat{D}_{\ab}$ obtained by applying $F$ to objects and structure maps. For example, the addition structure map $\mu_{FX}$ of $FX$ is the composite $FX \x FX \cong F(X \x X) \ral{F \mu_X} FX$.
\end{enumerate}
\end{lemma}

\begin{proof}
See \cite{Barr02}*{Chapter 2, \S 1}, along with a straightforward verification.
\end{proof}

\begin{corollary}
Any left adjoint or right adjoint between additive categories is additive.
\end{corollary}

\begin{definition} \label{def:PassesToAb}
We say that a functor $F \colon \cat{C} \to \cat{D}$ \Def{passes to abelian group objects} if there is an additive functor $F \colon \cat{C}_{\ab} \to \cat{D}_{\ab}$ making the diagram:
\[
\xymatrix{
\cat{C}_{\ab} \ar[d]_U \ar[r]^F & \cat{D}_{\ab} \ar[d]^U \\
\cat{C} \ar[r]^F & \cat{D} \\
}
\]
commute (up to natural isomorphism). In other words, given any abelian group object $A$ in $\cat{C}$, the object $FA$ in $\cat{D}$ can be endowed with the structure of an abelian group object, functorially in $A$.

We say that $F$ \Def{passes to Beck modules} if for every object $c$ of $\cat{C}$, the functor between slice categories $F \colon \cat{C}/c \to \cat{D}/Fc$ passes to abelian group objects.
\end{definition}

\begin{lemma}
Let $\cat{C}$ and $\cat{D}$ be categories with finite powers.
\begin{enumerate}
\item A functor $F \colon \cat{C} \to \cat{D}$ passes to abelian group objects if and only if it preserves finite powers of objects in the (essential) image of $U \colon \cat{C}_{\ab} \to \cat{C}$, i.e., those objects of $\cat{C}$ that can be endowed with an abelian group object structure.
\item If $F$ passes to abelian group objects, then the functor $F \colon \cat{C}_{\ab} \to \cat{D}_{\ab}$ is naturally isomorphic to the functor obtained by applying $F$ to the objects and structure maps.
\end{enumerate}
\end{lemma}

\begin{proof}
1. The ``if'' direction follows from the argument in \hyperref[InducedFctAb]{\ref*{Additive} (\ref*{InducedFctAb})}. The ``only if'' direction follows from the fact that $U \colon \cat{C}_{\ab} \to \cat{C}$ creates limits, and an additive functor $F \colon \cat{C}_{\ab} \to \cat{D}_{\ab}$ preserves finite powers.

2. Recall that the structure maps of an abelian group object $A$ are themselves maps in $\cat{C}_{\ab}$ (which would not be true, say, for group objects). By the interchange law, the addition structure map $\mu_{FA} \colon FA \x FA \to FA$ must be that induced by $F \mu_A$, and likewise for the remaining structure maps.
\end{proof}

\begin{example}
If $F$ preserves finite powers of objects admitting a map from the terminal object, then $F$ passes to abelian group objects. Likewise, if $F$ preserves limits of the form $E \x_c E \x_c \ldots \x_c E$ for every split epimorphism $p \colon E \to c$ in $\cat{C}$, then $F$ passes to Beck modules. Note that for a Beck module $p \colon E \to c$ over $c$, the projection map $p$ is a split epimorphism, because of the zero section $e \colon c \to E$.
\end{example}

\begin{proposition} \label{InducedAdjAb}
Let $F \colon \cat{C} \rla \cat{D} \colon G$ be an adjunction, and consider the induced functor $G \colon \cat{D}_{\ab} \to \cat{C}_{\ab}$ on abelian group objects.
\begin{enumerate}
\item \label{AdjLocPres} If $\cat{C}$ and $\cat{D}$ are locally presentable, then $G \colon \cat{D}_{\ab} \to \cat{C}_{\ab}$ has a left adjoint $\ind{F} \colon \cat{C}_{\ab} \to \cat{D}_{\ab}$.
\item \label{PassesToAb} If $F \colon \cat{C} \to \cat{D}$ passes to abelian group objects, then the induced functor $F \colon \cat{C}_{\ab} \to \cat{D}_{\ab}$ is left adjoint to $G \colon \cat{D}_{\ab} \to \cat{C}_{\ab}$. In particular, the two functors $F$ commute with the forgetful functors $U$, i.e., there is a natural isomorphism $F U = U F \colon \cat{C}_{\ab} \to \cat{D}$, as illustrated in the diagram:
\[
\xymatrix @R=3pc @C=3pc {
\cat{C} \ar@<-0.6ex>[d]_{F} \ar@<0.6ex>[r]^-{Ab} & \cat{C}_{\ab} \ar@<0.6ex>[l]^-{U} \ar@<-0.6ex>[d]_{\ind{F} = F} \\
\cat{D} \ar@<-0.6ex>[u]_{G} \ar@<0.6ex>[r]^{Ab} & \cat{D}_{\ab}. \ar@<0.6ex>[l]^{U} \ar@<-0.6ex>[u]_{G}
}
\]
\end{enumerate}
In view of part \ref{PassesToAb}, we will say that the left adjoint $\ind{F} \colon \cat{C}_{\ab} \to \cat{D}_{\ab}$ is \emph{induced} by $F \colon \cat{C} \to \cat{D}$, whether or not $F$ passes to abelian group objects.
\end{proposition}

\begin{proof}
1. Both $\cat{C}_{\ab}$ and $\cat{D}_{\ab}$ are locally presentable, by \ref{AbExists}, and $G \colon \cat{D}_{\ab} \to \cat{C}_{\ab}$ preserves limits. Thus it suffices to show that this functor is accessible, i.e., preserves $\ka$-filtered colimits for some regular cardinal $\ka$.

Let $\cat{C}$ be locally $\ka_1$-presentable and let $\cat{D}$ be locally $\ka_2$-presentable. By the adjoint functor theorem for locally presentable categories, the right adjoint $G \colon \cat{D} \to \cat{C}$ preserves $\ka_3$-filtered colimits for some regular cardinal $\ka_3$. The induced functor $G \colon \cat{D}_{\ab} \to \cat{C}_{\ab}$ makes the diagram
\[
\xymatrix{
\cat{D}_{\ab} \ar[d]_{U} \ar[r]^G & \cat{C}_{\ab} \ar[d]^{U} \\
\cat{D} \ar[r]^G & \cat{C} \\
}
\]
commute. By \ref{AbExists}, $U \colon \cat{C}_{\ab} \to \cat{C}$ creates $\ka_1$-filtered colimits and $U \colon \cat{D}_{\ab} \to \cat{D}$ preserves (in fact creates) $\ka_2$-filtered colimits. Therefore, $G \colon \cat{D}_{\ab} \to \cat{C}_{\ab}$ preserves $\ka$-filtered colimits for any regular cardinal $\ka$ greater than $\ka_1$, $\ka_2$, and $\ka_3$.

2. Consider abelian group objects $c$ in $\cat{C}_{\ab}$ and $d$ in $\cat{D}_{\ab}$. We want to exhibit a natural isomorphism:
\[
\xymatrix{
\Hom_{\cat{D}_{\ab}}(Fc,d) \ar@{^{(}->}[d] \ar@{==}[r]^{?} & \Hom_{\cat{C}_{\ab}}(c,Gd) \ar@{^{(}->}[d] \\
\Hom_{\cat{D}}(UFc,Ud) \ar@{=}[r] & \Hom_{\cat{C}}(Uc,UGd) \\
}
\]
which says that a map $f \colon Fc \to d$ in $\cat{D}$ respects the abelian group object structure if and only if its adjunct $f' \colon c \to Gd$ in $\cat{C}$ does. This holds, since the diagram
\[
\xymatrix{
F(c \x c) \ar[dr]_{F \mu_c} \ar[r]^-{\cong} & Fc \x Fc \ar[d]^{\mu_{Fc}} \ar[r]^-{f \x f} & d \x d \ar[d]^{\mu_d} \\
& Fc \ar[r]^f & d \\
}
\]
commutes if and only if its adjoint diagram:
\[
\xymatrix{
c \x c \ar[d]_{\mu_c} \ar[r]^-{f' \x f'} & Gd \x Gd \ar[d]_{\mu_{Gd}} & G(d \x d) \ar[dl]^{G \mu_d} \ar[l]_-{\cong} \\
c \ar[r]^{f'} & Gd & \\
}
\]
commutes, and likewise for the remaining structure maps.
\end{proof}

\subsection{Quasi-algebraic Beck modules}

Our next goal is to show that if $\cat{C}$ is quasi-algebraic, then its categories of Beck modules $(\cat{C}/X)_{\ab}$ are also quasi-algebraic, and therefore have nice simplicial objects.

\begin{proposition} \label{AlgebraicSlice}
Let $\cat{C}$ be a quasi-algebraic category and $X$ an object of $\cat{C}$. Then the slice category $\cat{C}/X$ is quasi-algebraic.
\end{proposition}

\begin{proof}
1. $\cat{C}/X$ has small colimits, since they are created by the forgetful functor $\cat{C}/X \to \cat{C}$.

2. Let $\cat{G}$ be a set of finitely presentable projective generators for $\cat{C}$. Then
\[
\left\{ P \ral{p} X \mid P \in \cat{G}, p \in \Hom_{\cat{C}}(P,X) \right\}
\]
is a set of finitely presentable projective generators for $\cat{C}/X$. Presentability is a straightforward verification; the rest follows from \ref{ProjSlice}, \ref{EpiSlice}, and the fact that $(\amalg P_i) \to X$ is the coproduct $\amalg (P_i \to X)$ in $\cat{C}/X$. (By the same argument, if $\cat{C}$ has enough projectives, then so does $\cat{C}/X$.)
\end{proof}

\begin{proposition} \label{AbCVariety}
If $\cat{C}$ is a many-sorted finitary variety (resp. quasivariety), then so is $\cat{C}_{\ab}$.

In particular, if $\cat{C}$ is algebraic (resp. quasi-algebraic), then so is $\cat{C}_{\ab}$.
\end{proposition}

\begin{proof}
Let $\cat{C} \cong \Alg(\Si,E,I)$ be an equational presentation of $\cat{C}$, where $\Si$ is an $S$-sorted signature for some set $S$. We treat the case of varieties simultaneously, by allowing the set of implications $I$ to be empty.

Objects of $\cat{C}_{\ab}$ have the underlying $S$-graded set of their underlying object in $\cat{C}$, equipped with the additional structure maps $\mu \colon X \x X \to X$, $e \colon \ast \to X$, and $\io \colon X \to X$ satisfying the conditions of associativity and so on, and the conditions that the structure maps be maps in $\cat{C}$. Since the forgetful functor $\cat{C} \to \Set^S$ creates limits, the structure maps amount to maps $\mu_s \colon X_s \x X_s \to X_s$, $e_s \colon \ast_s \to X_s$, and $\io_s \colon X_s \to X_s$ for each sort $s \in S$, making $X_s$ into an abelian group. Define the $S$-sorted signature
\[
\Si' \dfn \Si \cup \bigcup_{s \in S} \{ \mu_s, e_s, \io_s \}
\]
where the additional operations $\mu_s, e_s, \io_s$ have arities $s \x s \to s$, $\to s$, and $s \to s$ respectively. Let $E_{\ab,s}$ denote the set of equations for $\mu_s, e_s, \io_s$ as abelian group structure maps, as in \ref{AbGpsVar}, and let $E_{\ab} \dfn \bigcup_{s \in S} E_{\ab,s}$. Recall that $\cat{C}$ is a full subcategory of $\Alg \Si$, that is, maps in $\cat{C}$ are maps of underlying $S$-graded sets that commute with all operations. For each operation symbol $\si \in \Si$ of arity $s_1 \x \ldots \x s_n \to s$, consider the condition that $\mu \colon X \x X \to X$ commute with $\si$:
\[
\xymatrix{
(X \x X)_{s_1} \x \ldots \x (X \x X)_{s_n} \ar[d]_{\mu_{s_1} \x \ldots \x \mu_{s_n}} \ar[r]^-{\si_{X \x X}} & (X \x X)_s \ar[d]^{\mu_s} \\
X_{s_1} \x \ldots \x X_{s_n} \ar[r]^-{\si_X} & X_s. \\
}
\]
Let $E_{\si}$ denote the set containing the three equations expressing the compatibility of $\si$ with $\mu$, $e$, and $\io$ respectively, and let $E_{\mathrm{struc}} \dfn \bigcup_{\si \in \Si} E_{\si}$. With the set of equations $E' \dfn E \cup E_{\ab} \cup E_{\mathrm{struc}}$, we obtain the equational presentation $C_{\ab} \cong \Alg (\Si', E', I)$.
\end{proof}

The proof of  \ref{AbCVariety} does not provide an explicit set of finitely presentable projective generators for $\cat{C}_{\ab}$. For the record, we describe such a set below in \ref{GeneratorsAbC}.

\begin{lemma} \label{AdjRegEpi}
Assume $\cat{C}$ is regular and $U \colon \cat{C}_{\ab} \to \cat{C}$ has a left adjoint. If a map $f \colon X \surj UB$ is a regular epimorphism in $\cat{C}$, then its adjunct map $f' \colon Ab X \to B$ is a regular epimorphism in $\cat{C}_{\ab}$. In particular, the counit $Ab U A \surj A$ is always a regular epimorphism.
\end{lemma}

\begin{proof}
Recall that $Ab X \to B$ is a regular epimorphism in $\cat{C}_{\ab}$ if and only if $U Ab X \to UB$ is a regular epimorphism in $\cat{C}$. The regular epimorphism $f$ factors as $f = (Uf') \circ \eta_X \colon X \to U Ab X \to UB$, which implies that $Uf'$ is a regular epimorphism since $\cat{C}$ is regular \cite{Borceux94v1}*{Corollary 2.1.5 (2)}.
\end{proof}

\begin{remark}
The converse is false in general. For example, take $\cat{C} = \Set$, $X = \{*\}$, $Y = \Z$, and $f(*) = 1$. The map $f$ is far from being a regular epimorphism (i.e., a surjection), but its adjunct $f' \colon Ab(*) = \Z \to \Z$ is a regular epimorphism, even an isomorphism.
\end{remark}

\begin{lemma} \label{ProjAbC}
Assume $\cat{C}$ is regular and has enough projectives, and $U \colon \cat{C}_{\ab} \to \cat{C}$ has a left adjoint. Then an object of $\cat{C}_{\ab}$ is projective if and only if it is a retract of $Ab P$ for some projective $P$ of $\cat{C}$.
\end{lemma}

\begin{proof}
($\La$) Trying to lift a map $Ab P \to B$ along a regular epimorphism $A \surj B$ is the same as trying to lift the adjunct map:
\[
\xymatrix{
& P \ar@{-->}[dl] \ar[d] \\
UA \ar@{->>}[r] & UB. \\
}
\]
The bottom map is a regular epimorphism since $U$ preserves them, and thus the lift exists. Therefore $Ab P$ is projective, and a retract of a projective is projective.

($\Ra$) Let $Q$ be a projective in $\cat{C}_{\ab}$. Since $\cat{C}$ has enough projectives, there is a projective $P$ of $\cat{C}$ with a regular epimorphism $P \surj UQ$. Take its adjunct map $Ab P \surj Q$, which is still a regular epimorphism by \ref{AdjRegEpi}. Lifting the identity of $Q$ along that regular epimorphism exhibits $Q$ as a retract of $Ab P$.
\end{proof}

\begin{proposition} \label{GeneratorsAbC}
Let $\cat{C}$ be a quasi-algebraic category and $S$ a set of finitely presentable projective generators for $\cat{C}$. Then $\left\{ Ab P \mid P \in S \right\}$ is a set of finitely presentable projective generators for $\cat{C}_{\ab}$.
\end{proposition}

\begin{proof}
Since $\cat{C}$ is locally finitely presentable, $U \colon \cat{C}_{\ab} \to \cat{C}$ has a left adjoint $Ab \colon \cat{C} \to \cat{C}_{\ab}$, by \hyperref[AbExistsPart]{\ref*{AbExists} (\ref*{AbExistsPart})}. Each $Ab P$ is finitely presentable, by \hyperref[CreateFiltColim]{\ref*{AbExists} (\ref*{CreateFiltColim})}, and projective, by \ref{ProjAbC}. Let us show that they form a family of generators. For any object $X$ of $\cat{C}_{\ab}$, take a regular epimorphism $\amalg P_i \surj UX$ from a coproduct of generators in $S$. Then the adjunct map $\amalg Ab(P_i) = Ab \left( \amalg P_i \right) \surj X$ is a regular epimorphism in $\cat{C}_{\ab}$, by \ref{AdjRegEpi}. (By the same argument, if $\cat{C}$ has enough projectives, then so does $\cat{C}_{\ab}$.)
\end{proof}

\begin{remark}
Proposition \ref{GeneratorsAbC} provides an alternate proof of Proposition \ref{AbCVariety} which does not rely on the universal-algebraic characterization theorems. If $\cat{C}$ is quasi-algebraic, then $\cat{C}$ is in particular locally finitely presentable. By \hyperref[AbCLocPres]{\ref*{AbExists} (\ref*{AbCLocPres})}, $\cat{C}_{\ab}$ is also locally finitely presentable, in particular cocomplete. By \ref{GeneratorsAbC}, $\cat{C}_{\ab}$ has a set of finitely presentable projective generators, and is therefore quasi-algebraic. If moreover $\cat{C}$ is exact, then so is $\cat{C}_{\ab}$, by \cite{Barr02}*{Chapter 2, Theorem 2.3}, and thus $\cat{C}_{\ab}$ is algebraic, by \ref{AlgExact}.
\end{remark}

\subsection{The setup}

Putting the ingredients together, we obtain a good setup for Quillen cohomology. It is essentially an observation of Quillen \cite{Quillen67}*{\S II.5, (4) before Theorem 5}, which we state and prove in more detail.

\begin{proposition} \label{GoodSetupHQ}
Let $\cat{C}$ be a quasi-algebraic category and $X$ an object of $\cat{C}$. Then $\cat{C}/X$ and $(\cat{C}/X)_{\ab}$ are quasi-algebraic, and in the prolonged adjunction:
\[
\xymatrix{
s\cat{C}/X \ar@<0.6ex>[r]^-{Ab_X} & s (\cat{C}/X)_{\ab} \ar@<0.6ex>[l]^-{U_X} \\
}
\]
the right adjoint $U_X \colon s(\cat{C}/X)_{\ab} \to s\cat{C}/X$ creates (i.e., preserves and reflects) fibrations and weak equivalences. In particular, the prolonged adjunction is a Quillen pair.
\end{proposition}

\begin{proof}
The abelianization $Ab_X \colon \cat{C}/X \to (\cat{C}/X)_{\ab}$ exists, by \ref{AbExist}. Both $\cat{C}/X$ and $(\cat{C}/X)_{\ab}$ are quasi-algebraic, by \ref{AlgebraicSlice} and \ref{AbCVariety}; in particular, both have nice simplicial objects. By definition, a map $f \colon M_{\bu} \to N_{\bu}$ in $s (\cat{C}/X)_{\ab}$ is a fibration (resp. weak equivalence) if the map of simplicial sets:
\begin{equation} \label{eq:HomFromProj}
\xymatrix{
\Hom_{(\cat{C}/X)_{\ab}}(P', M_{\bu}) \ar[r]^{f_*} & \Hom_{(\cat{C}/X)_{\ab}}(P', N_{\bu}) \\
}
\end{equation}
is so, for every projective $P'$ in $(\cat{C}/X)_{\ab}$. By \ref{GeneratorsAbC}, it suffices that the condition hold for projectives of the form $P' = Ab_X P$, where $P$ is a projective in $\cat{C}/X$. In that case, the map \eqref{eq:HomFromProj} is a fibration (resp. weak equivalence) of simplicial sets if and only if its adjunct map:
\[
\xymatrix{
\Hom_{\cat{C}/X}(P, U_X M_{\bu}) \ar[r]^{U_X f_*} & \Hom_{\cat{C}/X}(P, U_X N_{\bu}) \\
}
\]
is so. This holds for every projective $P$ in $\cat{C}/X$ if and only if $U_X f \colon U_X M_{\bu} \to U_X N_{\bu}$ is a fibration (resp. weak equivalence) in $s (\cat{C}/X)$, or equivalently in $(s \cat{C})/X$, by \ref{SameModelStruc}.
\end{proof}


The setup above is not quite enough to work with Quillen cohomology. In \cite{Quillen67}*{\S II.5}, Quillen describes additional assumptions on the homotopy category $\Ho (s\cat{C}/X_{\bu})_{\ab}$, which are satisfied for example if $\cat{C}$ has abelian Beck modules, i.e., the category $(\cat{C}/X)_{\ab}$ is abelian for every object $X$. One condition guaranteeing abelian Beck modules is exactness \cite{Barr02}*{Chapter 2, Theorem 2.4}, though this condition is not necessary, as we show below.

\begin{example} \label{NonAbelBeck}
A quasi-algebraic category does not necessarily have abelian Beck modules. For example, take the category $\tfAb$ of torsion-free abelian groups, viewed as a full subcategory of abelian groups, with inclusion $\io \colon \tfAb \to \Ab$.

Let us show that $\tfAb$ does not have abelian Beck modules. Since $\io$ preserves limits, a Beck module $E \to G$ over a torsion-free abelian group $G$ is in particular a Beck module viewed in $\Ab$, i.e., a direct sum $G \op M \surj G$. The only additional condition is that $G \op M$ be torsion-free, which happens if and only if $M$ itself is torsion-free. Hence, for every object $G$, we have $(\tfAb/G)_{\ab} \cong \tfAb$, which is not an abelian category.
\end{example}


\begin{example} \label{NonExact}
A quasi-algebraic category with abelian Beck modules is not necessarily exact. For example, take the category $\redCom$ of reduced commutative rings, viewed as a full subcategory of all commutative (unital) rings, with inclusion $\io \colon \redCom \to \Com$.

Let us show that $\redCom$ has abelian Beck modules. A Beck module over a reduced commutative ring $R$ is in particular a Beck module viewed in $\Com$, i.e., a square zero extension $R \op M \surj R$ with multiplication $(r,m)(r',m') = (rr', rm' + mr')$, where the left and right actions of $R$ on $M$ coincide. The only additional condition is for $R \op M$ to be a reduced ring, which happens if and only if $M$ is zero, since the nilradical is $\Nil(R \op M) = M$. Hence for every object $R$, we have $(\redCom/R)_{\ab} \cong 0$, which is an abelian category.
\end{example}

In short, a quasi-algebraic category has most of the ingredients for Quillen cohomology. An algebraic category (i.e., a quasi-algebraic category which is exact) has all the ingredients.


\section{Effect of an adjunction} \label{sec:Effect}

In this section, we investigate the main question: What does an adjunction $F \colon \cat{C} \rla \cat{D} \colon G$ do to Quillen (co)homology?

\textbf{Assumptions for Section \ref{sec:Effect}:} $\cat{C}$ and $\cat{D}$ are quasi-algebraic categories with abelian Beck modules. This is satisfied in particular when $\cat{C}$ and $\cat{D}$ are algebraic categories.

\subsection{Effect on Beck modules}

First, let us see how an adjunction passes to slice categories. There are two versions, depending if one starts with a ground object in $\cat{C}$ or in $\cat{D}$.  A straightforward verification yields the following proposition.

\begin{proposition} \label{AdjSlice}
\begin{enumerate}
\item For every object $c$ in $\cat{C}$, there is an induced adjunction:
\begin{equation} \label{AdjSliceC}
\xymatrix{
\cat{C}/c \ar@<0.6ex>[r]^-{F} & \cat{D}/Fc \ar@<0.6ex>[l]^-{\eta_c^* G} \\
}
\end{equation}
where $\eta_c \colon c \to GFc$ is the unit map.
\item For every object $d$ in $\cat{D}$, there is an induced adjunction:
\[
\xymatrix{
\cat{C}/Gd \ar@<0.6ex>[r]^-{\ep_{d!} F} & \cat{D}/d \ar@<0.6ex>[l]^-{G} \\
}
\]
where $\ep_d \colon FGd \to d$ is the counit map.
\end{enumerate}
\end{proposition}

\begin{proposition} \label{AdjBeckMod}
\begin{enumerate}
\item For every object $c$ in $\cat{C}$, there is an induced adjunction on Beck modules:
\[
\xymatrix{
(\cat{C}/c)_{\ab} \ar@<0.6ex>[r]^-{\ind{F}} & (\cat{D}/Fc)_{\ab}. \ar@<0.6ex>[l]^-{\eta_c^* G}
}
\]
\item For every object $d$ in $\cat{D}$, there is an induced adjunction on Beck modules:
\[
\xymatrix{
(\cat{C}/Gd)_{\ab} \ar@<0.6ex>[r]^-{\ep_{d \push} \ind{F}} & (\cat{D}/d)_{\ab}. \ar@<0.6ex>[l]^-{G}
}
\]
\item If $F \colon \cat{C} \to \cat{D}$ passes to Beck modules, then the induced left adjoint on Beck modules is $\ind{F} = F \colon (\cat{C}/c)_{\ab} \to (\cat{D}/Fc)_{\ab}$, the functor obtained by applying $F$ to the objects and structure maps.
\end{enumerate}
\end{proposition}

\begin{proof}
1. By \ref{InducedAdjAb} and the fact that $\cat{C}$ and $\cat{D}$ are locally presentable, the adjunction \eqref{AdjSliceC} induces such an adjunction on abelian group objects.

2. Using the previous part, consider the two adjunctions:
\[
\xymatrix{
(\cat{C}/Gd)_{\ab} \ar@<0.6ex>[r]^-{\ind{F}} & (\cat{D}/FGd)_{\ab} \ar@<0.6ex>[r]^-{\ep_{d \push}} \ar@<0.6ex>[l]^-{\eta^*_{Gd}G} & (\cat{D}/d)_{\ab}. \ar@<0.6ex>[l]^-{\ep_{d}^*} \\
}
\]
The result follows from the natural isomorphism $\eta_{Gd}^* G \ep_d^* \cong G \colon (\cat{D}/d)_{\ab} \to (\cat{C}/Gd)_{\ab}$, which follows from the equality $\id_{Gd} = G \ep_d \circ \eta_{Gd} \colon Gd \to GFGd \to Gd$.

3. This follows from \ref{InducedAdjAb}. 
\end{proof}

\subsection{Effect on abelian cohomology} \label{sec:AdjHA}

Before introducing any homotopical algebra, let us study the problem at the level of homological algebra. We want to describe the effect of the adjunction on abelian cohomology. As we have seen in \ref{AdjBeckMod}, there are two induced adjunctions, depending if one starts with a ground object in $\cat{C}$ or in $\cat{D}$.

\renewcommand\thesubsubsection{\Roman{subsubsection}}
\subsubsection{Ground object in \texorpdfstring{$\cat{C}$}{C}}

Pick a ground object $c$ in $\cat{C}$. The induced adjunction on Beck modules fits into the diagram:
\begin{equation} \label{DiagAdjHAc}
\xymatrix @R=3pc @C=3pc {
\cat{C}/c \ar@<-0.6ex>[d]_{F} \ar@<0.6ex>[r]^-{Ab_c} & (\cat{C}/c)_{\ab} \ar@<0.6ex>[l]^-{U_c} \ar@<-0.6ex>[d]_{\ind{F}} \\
\cat{D}/Fc \ar@<-0.6ex>[u]_{\eta_c^* G} \ar@<0.6ex>[r]^{Ab_{Fc}} & (\cat{D}/Fc)_{\ab} \ar@<0.6ex>[l]^{U_{Fc}} \ar@<-0.6ex>[u]_{\eta_c^* G}
}
\end{equation}
where the diagram of right adjoints commutes (strictly), and thus the diagram of left adjoints commutes as well. In particular, applying the left adjoints to $\id_c$, we obtain $\ind{F} Ab_c c = Ab_{Fc} Fc$. Take a module $N$ over $Fc$ and consider:
\begin{align} 
\HA^*(c; \eta_c^* G N) &= \Ext^*(Ab_c c, \eta_c^* G N) \notag \\
&= \Hy^* \Hom_{\Mod_c} (P_{\bu}, \eta_c^* G N) \notag \\
&= \Hy^* \Hom_{\Mod_{Fc}} (\ind{F} P_{\bu}, N) \label{ComputHAc}
\end{align}
where $P_{\bu} \to Ab_c c$ is a projective resolution. We want to compare this to:
\begin{align*}
\HA^*(Fc; N) &= \Ext^*(Ab_{Fc} Fc, N) \\
&= \Hy^* \Hom_{\Mod_{Fc}} (Q_{\bu}, N)
\end{align*}
where $Q_{\bu} \to Ab_{Fc} Fc$ is a projective resolution. Assume the induced left adjoint $\ind{F} \colon (\cat{C}/c)_{\ab} \to (\cat{D}/Fc)_{\ab}$ preserves projectives (which is the case for example when its right adjoint $\eta_c^* G$ preserves epimorphisms, i.e., is exact). Then $\ind{F} P_{\bu}$ is projective but is not a resolution of $\ind{F} Ab_c c$. However, the map factors as $\ind{F} P_{\bu} \inj Q_{\bu} \ral{\sim} \ind{F} Ab_c c = Ab_{Fc} Fc$ and the first map induces
\[
\Hom_{\Mod_{Fc}} (Q_{\bu},N) \to \Hom_{\Mod_{Fc}} (\ind{F} P_{\bu},N)
\]
which, upon passing to cohomology, induces a well-defined map. We sum up the argument in the following proposition.

\begin{proposition} 
If the left adjoint $F$ induces a functor $\ind{F}$ on Beck modules which preserves projectives, then we obtain a comparison map in abelian cohomology:
\begin{equation} \label{ComparMapHAc}
\HA^*(Fc; N) \to \HA^*(c; \eta_c^* G N).
\end{equation}
\end{proposition}

Note that \eqref{ComputHAc} exhibits $\HA^*(c; \eta_c^* G N)$ as the derived functors of $\Hom_{\Mod_{Fc}}(-,N) \circ \ind{F}$ applied to $Ab_c c$. Since $\ind{F}$ sends projectives to projectives, we obtain a Grothendieck spectral sequence:
\[
E_2^{s,t} = \Ext^s \left( (L_t \ind{F})(Ab_c c), N \right) \Ra \HA^{s+t}(c; \eta_c^* G N)
\]
which is first quadrant, cohomologically graded. The comparison map \eqref{ComparMapHAc} is the edge morphism:
\[
\HA^s (Fc; N) = \Ext^s (\ind{F} Ab_c c, N) = E_2^{s,0} \surj E_{\infty}^{s,0} \inj \HA^s (c; \eta_c^* G N).
\]
If $\ind{F} \colon (\cat{C}/c)_{\ab} \to (\cat{D}/Fc)_{\ab}$ happens to be exact, then $\ind{F} P_{\bu}$ is a projective resolution of $\ind{F} Ab_c c = Ab_{Fc} Fc$ and the comparison map \eqref{ComparMapHAc} is an isomorphism.

\begin{remark}
Starting with a module $M$ over $c$, there is a map:
\begin{align*}
\Hom_{\Mod_c} (Ab_c c, M) \to &\Hom_{\Mod_{Fc}} (\ind{F} Ab_c c, \ind{F} M) \\
= &\Hom_{\Mod_{Fc}} (Ab_c c, \eta_c^* G \ind{F}M)
\end{align*}
given by applying $\ind{F}$, or equivalently, induced by the unit $M \to \eta_c^* G \ind{F} M$. One might want to compare $\HA^*(c;M)$ and $\HA^*(Fc; \ind{F} M)$, but they both naturally map to $\HA^*(c; \eta_c^* G \ind{F} M)$, respectively via the unit and the comparison map \eqref{ComparMapHAc}. There is no direct comparison.
\end{remark}

\subsubsection{Ground object in \texorpdfstring{$\cat{D}$}{D}}

Pick a ground object $d$ in $\cat{D}$. The induced adjunction on Beck modules fits into the diagram:
\begin{equation} \label{DiagAdjHAd}
\xymatrix @R=3pc @C=3pc {
\cat{C}/Gd \ar@<-0.6ex>[d]_{\ep_{d!} F} \ar@<0.6ex>[r]^-{Ab_{Gd}} & (\cat{C}/Gd)_{\ab} \ar@<0.6ex>[l]^-{U_{Gd}} \ar@<-0.6ex>[d]_{\ep_{d \push} \ind{F}} \\
\cat{D}/d \ar@<-0.6ex>[u]_G \ar@<0.6ex>[r]^{Ab_d} & (\cat{D}/d)_{\ab} \ar@<0.6ex>[l]^{U_d} \ar@<-0.6ex>[u]_G
}
\end{equation}
where the diagram of right adjoints commutes, and thus the diagram of left adjoints commutes as well. Take a module $N$ over $d$ and consider:
\begin{align*}
\HA^*(d; N) &= \Ext^*(Ab_{d} d, N) \\
&= \Hy^* \Hom_{\Mod_{d}} (P_{\bu}, N)
\end{align*}
where $P_{\bu} \to Ab_d d$ is a projective resolution. We want to compare this to:
\begin{align}
\HA^*(Gd; G N) &= \Ext^*(Ab_{Gd} Gd, G N) \notag \\
&= \Hy^* \Hom_{\Mod_{Gd}} (Q_{\bu}, G N) \notag \\
&= \Hy^* \Hom_{\Mod_d} (\ep_{d \push} \ind{F} Q_{\bu}, N) \label{ComputHAd}
\end{align}
where $Q_{\bu} \to Ab_{Gd} Gd$ is a projective resolution. Here again, assume the induced left adjoint $\ep_{d \push} \ind{F} \colon (\cat{C}/Gd)_{\ab} \to (\cat{D}/d)_{\ab}$ preserves projectives. Then $\ep_{d \push} \ind{F} Q_{\bu}$ is projective and we have a map:
\begin{align*}
\ep_{d \push} \ind{F} Q_{\bu} \to &\ep_{d \push} \ind{F} Ab_{Gd} Gd \\
= &\ep_{d \push} Ab_{FGd} FGd \\
= &Ab_d (FGd \ral{\ep_d} d) \ral{Ab_d(\ep_d)} Ab_d d.
\end{align*}
It admits a factorization $\ep_{d \push} \ind{F} Q_{\bu} \inj P_{\bu} \ral{\sim} Ab_d d$ and the first map induces:
\[
\Hom_{\Mod_d} (P_{\bu},N) \to \Hom_{\Mod_d} (\ep_{d \push} \ind{F} Q_{\bu},N)
\]
which, upon passing to cohomology, induces a well-defined map. We sum up the argument in the following proposition.

\begin{proposition} 
If the induced left adjoint $\ep_{d \push} \ind{F} \colon (\cat{C}/Gd)_{\ab} \to (\cat{D}/d)_{\ab}$ preserves projectives, then we obtain a comparison map in abelian cohomology:
\begin{equation} \label{ComparMapHAd}
\HA^*(d; N) \to \HA^*(Gd; G N).
\end{equation}
\end{proposition}

Note that \eqref{ComputHAd} exhibits $\HA^*(Gd;GN)$ as the derived functors of $\Hom_{\Mod_d}(-,N) \circ \ep_{d \push} \ind{F}$ applied to $Ab_{Gd} Gd$. Since $\ep_{d \push} \ind{F}$ sends projectives to projectives, we obtain a Grothendieck composite spectral sequence:
\[
E_2^{s,t} = \Ext^s \left( L_t(\ep_{d \push} \ind{F}) (Ab_{Gd} Gd), N \right) \Ra \HA^{s+t}(Gd; GN)
\]
which is first quadrant, cohomologically graded. The comparison map \eqref{ComparMapHAd} is $Ab_d(\ep_d)^*$ followed by an edge morphism:
\begin{align*}
\HA^s (d; N) = \Ext^s (Ab_d d, N) \ral{Ab_d(\ep_d)^*} &\Ext^s (\ep_{d \push} \ind{F} Ab_{Gd} Gd, N) \\
= &E_2^{s,0} \surj E_{\infty}^{s,0} \inj \HA^s (Gd;GN).
\end{align*}
If $\ep_{d \push} \ind{F} \colon (\cat{C}/Gd)_{\ab} \to (\cat{D}/d)_{\ab}$ happens to be exact, then $\ep_{d \push} \ind{F} Q_{\bu}$ is a projective resolution of $\ep_{d \push} \ind{F} Ab_{Gd} Gd$, and we obtain an isomorphism $\Ext^* (\ep_{d \push} \ind{F} Ab_{Gd} Gd, N) \cong \HA^* (Gd;GN)$. In that case, the comparison map \eqref{ComparMapHAd} is simply $Ab_d(\ep_d)^*$, which is not necessarily an isomorphism.

\begin{remark}
Starting with a module $M$ over $Gd$, one might want to compare $\HA^*(Gd;M)$ and $\HA^*(d; \ep_{d \push} \ind{F} M)$. Again, there is no direct comparison. They both map naturally to \linebreak $\HA^*(Gd; G \ep_{d \push} \ind{F} M)$, the former via the unit $M \to G \ep_{d \push} \ind{F} M$ and the latter via the comparison map \eqref{ComparMapHAd}.
\end{remark}

\subsection{The comparison diagram}

Now let us check that the adjunction behaves well at the level of homotopical algebra, when we pass to simplicial objects.

\begin{theorem} \label{ComparDiag}
Let $\cat{C}$ and $\cat{D}$ be quasi-algebraic categories with abelian Beck modules. Let $F \colon \cat{C} \rla \cat{D} \colon G$ be an adjunction that prolongs to a Quillen pair (equivalently, $G$ preserves regular epimorphisms, or $F$ preserves projectives). Then for every object $c$ of $\cat{C}$, the commutative diagram \eqref{DiagAdjHAc} simplicially prolongs to four Quillen pairs:
\[
\xymatrix @R=3pc @C=3pc {
s \cat{C}/c \ar@<-0.6ex>[d]_{F} \ar@<0.6ex>[r]^-{Ab_c} & s (\cat{C}/c)_{\ab} \ar@<0.6ex>[l]^-{U_c} \ar@<-0.6ex>[d]_{\ind{F}} \\
s \cat{D}/Fc \ar@<-0.6ex>[u]_{\eta_c^* G} \ar@<0.6ex>[r]^{Ab_{Fc}} & s (\cat{D}/Fc)_{\ab} \ar@<0.6ex>[l]^{U_{Fc}} \ar@<-0.6ex>[u]_{\eta_c^* G}
}
\]
where moreover, the right Quillen functors $U_c$ and $U_{Fc}$ create fibrations and weak equivalences.

Likewise, for every object $d$ of $\cat{D}$, the commutative diagram \eqref{DiagAdjHAd} simplicially prolongs to four Quillen pairs:
\[
\xymatrix @R=3pc @C=3pc {
s \cat{C}/Gd \ar@<-0.6ex>[d]_{\ep_{d!} F} \ar@<0.6ex>[r]^-{Ab_{Gd}} & s (\cat{C}/Gd)_{\ab} \ar@<0.6ex>[l]^-{U_{Gd}} \ar@<-0.6ex>[d]_{\ep_{d \push} \ind{F}} \\
s \cat{D}/d \ar@<-0.6ex>[u]_G \ar@<0.6ex>[r]^{Ab_d} & s (\cat{D}/d)_{\ab} \ar@<0.6ex>[l]^{U_d} \ar@<-0.6ex>[u]_G
}
\]
where $U_d$ and $U_{Gd}$ create fibrations and weak equivalences.
\end{theorem}

\begin{proof}
The statements about the rows of the diagrams follow from \ref{GoodSetupHQ}. Now we prove the rest.

\emph{Case 1: Ground object $c$ in $\cat{C}$.} The induced right adjoint on slice categories is \linebreak $\eta_c^* G \colon \cat{D}/Fc \to \cat{C}/c$ and it preserves regular epimorphisms. Indeed, $G \colon \cat{D}/Fc \to \cat{C}/GFc$ preserves regular epimorphisms by assumption and \ref{EpiSlice}. The pullback $\eta_c^*$ also preserves regular epimorphisms since $\cat{C}$ is regular and again by \ref{EpiSlice}.

The induced right adjoint on Beck modules $\eta_c^* G \colon (\cat{D}/Fc)_{\ab} \to (\cat{C}/c)_{\ab}$ preserves regular epimorphisms. This follows from the same argument, and the fact that regular epimorphisms in $(-)_{\ab}$ are preserved and reflected by the forgetful functor $U$, by  \ref{PreserveRegEpi} and \ref{ReflectRegEpi}.

\emph{Case 2: Ground object $d$ in $\cat{D}$.} The induced right adjoint on slice categories is just $G \colon \cat{D}/d \to \cat{C}/Gd$, which preserves regular epimorphisms. The induced right adjoint on Beck modules $G \colon (\cat{D}/d)_{\ab} \to (\cat{C}/Gd)_{\ab}$ also preserves regular epimorphisms.
\end{proof}

\subsection{Effect on Quillen (co)homology} \label{sec:AdjHQ}

In this section, we describe the comparison maps induced on Quillen (co)homology. The argument is similar to Section  \ref{sec:AdjHA}, except that we start with the comparison diagrams in \ref{ComparDiag}.

\begin{definition}
For every object $c$ of $\cat{C}$ and module $M$ in $(\cat{C}/c)_{\ab}$, the composite:
\[
\xymatrix{
\cat{C}/c \ar[r]^-{Ab_c} & (\cat{C}/c)_{\ab} \ar[r]^-{\Hom(-,M)} & \Ab^{\opp} \\
}
\]
evaluated at $\id_c$ gives rise to a composite spectral:
\begin{equation} \label{eq:UCSS}
E_2^{s,t} = \Ext^s \left( \HQ_t(c), M \right) \Ra \HQ^{s+t}(c;M) 
\end{equation}
which is first quadrant, cohomologically graded. We call it the \Def{universal coefficient spectral sequence} for Quillen cohomology. It has a left edge morphism:
\[
\HQ^t(c;M) \surj E_{\infty}^{0,t} \inj E_2^{0,t} = \Hom_{\Mod_c} \left( \HQ_t(c), M \right) 
\]
which is the pairing between homology and cohomology, and a bottom edge morphism:
\[
\HA^s(c;M) = E_2^{s,0} \surj E_{\infty}^{s,0} \inj \HQ^s(c;M)
\]
which is the comparison induced by the map $\LL_c \to Ab_c c$ in $s(\cat{C}/c)_{\ab}$.
\end{definition}

\begin{remark}
Since $\cat{C}/c$ is not an abelian category, the spectral sequence \eqref{eq:UCSS} is not a Gro\-then\-dieck spectral sequence in the usual sense \cite{Weibel94}*{Corollary 5.8.4}, but rather an instance of the related hyperhomology spectral sequence \cite{Weibel94}*{Proposition 5.7.6, Corollary 5.7.7}.
\end{remark}

\subsubsection{Ground object in \texorpdfstring{$\cat{C}$}{C}}

\begin{proposition} \label{EffectQuilHomolC}
Assume the setup of \ref{ComparDiag}. Then the comparison diagram induces the following comparison maps.
\begin{enumerate}
\item A natural (up to homotopy) comparison map of cotangent complexes:
\begin{equation} \label{eq:ComparCotCpxC}
\al \colon \ind{F} (\LL_c) \to \LL_{Fc}.
\end{equation}
\item For every degree $n \geq 0$, a natural comparison map in Quillen homology:
\begin{equation} \label{eq:ComparQuilHomolC}
\ind{F} \left( \HQ_n(c) \right) \to \HQ_n(Fc)
\end{equation}
which factors as a composite:
\[
\xymatrix{
\ind{F} \left( \HQ_n(c) \right) \ar[r]^-{h} & L_n (\ind{F} Ab_c)(\id_c) \ar[r]^-{\pi_n (\al)} & \HQ_n(Fc)  
}
\]
where $h$ is a left edge morphism in a composite spectral sequence:
\begin{equation} \label{eq:GrothSS}
E^2_{s,t} = (L_s \ind{F}) \left( \HQ_t(c) \right) \Ra L_{s+t}(\ind{F} Ab_c)(\id_c) = \pi_{s+t} (\ind{F} \LL_c).
\end{equation}
\item If $F$ preserves pullbacks, then the map $h \colon \ind{F} \left( \HQ_n(c) \right) \ral{\cong} \pi_n \ind{F} (\LL_c)$ is an isomorphism, so that the map \eqref{eq:ComparQuilHomolC} can be identified with $\pi_n (\al)$, the effect of \eqref{eq:ComparCotCpxC} on homotopy.
\item If $F$ preserves all weak equivalences, then \eqref{eq:ComparCotCpxC} is a weak equivalence. If moreover $F$ passes to Beck modules, then \eqref{eq:ComparQuilHomolC} is an isomorphism.
\end{enumerate}
\end{proposition}

\begin{proof}
1. Starting with a cofibrant replacement $q_c \colon Qc \ral{\sim} c$ of $\id_c$, we can apply $F$ to obtain $FQc \to Fc$, where the source is still cofibrant (since $F$ is a left Quillen functor) but the map is not a weak equivalence anymore. However, it factors (uniquely and functorially up to homotopy) as $FQc \ral{\psi} QFc \ral{\sim} Fc$ and we obtain the comparison map:
\begin{align*}
\ind{F} (\LL_c) &= \ind{F} Ab_c (Qc \to c) \\
&= Ab_{Fc} F (Qc \to c) \\
&= Ab_{Fc} (FQc \to Fc) \to Ab_{Fc} (QFc \to Fc) = \LL_{Fc}
\end{align*}
which is $Ab_{Fc}(\psi)$.

2. Consider the composite of left adjoints $\cat{C}/c \ral{Ab_c} (\cat{C}/c)_{\ab} \ral{\ind{F}} (\cat{D}/Fc)_{\ab}$ where the categories of Beck modules $(\cat{C}/c)_{\ab}$ and $(\cat{D}/Fc)_{\ab}$ are abelian by assumption, and $\ind{F}$ is additive, by \ref{Additive}. Both $\cat{C}/c$ and $(\cat{C}/c)_{\ab}$ have enough projectives, since they are quasi-algebraic. Since $Ab_c$ prolongs to a left Quillen functor, we obtain a first quadrant, homologically graded composite spectral sequence:
\[
E^2_{s,t} = (L_s \ind{F}) \left( L_t Ab_c \right)(E \to c) \Ra L_{s+t}(\ind{F} Ab_c)(E \to c)
\]
for any object $E \to c$ of $\cat{C}/c$. Applying the spectral sequence to $\id_c$ yields the $E^2$ term in \eqref{eq:GrothSS}. The left edge morphism is:
\[
\ind{F} \HQ_t(c) = E^2_{0,t} \surj E^{\infty}_{0,t} \inj L_t (\ind{F} Ab_c)(\id_c).
\]
This edge morphism can also be described as the homology comparison map \cite{Barr06}*{Theorems 2.2 and 2.6} for the right exact functor $\ind{F}$, applied to the chain complex $\LL_c$ (using implicitly the Dold-Kan correspondence):
\[
\ind{F} \left( \HQ_n(c) \right) = \ind{F} H_n(\LL_c) \to H_n \ind{F} (\LL_c).
\]

3. If $F$ preserves pullbacks, then $F$ passes to Beck modules and moreover, the induced left adjoint $\ind{F} = F \colon (\cat{C}/c)_{\ab} \to (\cat{D}/Fc)_{\ab}$ preserves finite limits, hence is left exact (and thus exact). In that case, the homology comparison $h$ is an isomorphism.

4. If $F$ preserves all weak equivalences, then the map $\psi \colon FQc \ral{\sim} QFc$ is a weak equivalence (between cofibrant objects). Since $Ab_{Fc}$ is a left Quillen functor, the map $\al = Ab_{Fc}(\psi)$ is also a weak equivalence.

If moreover $F$ passes to Beck modules, then the induced left adjoint $\ind{F} = F$ also preserves all weak equivalences, and in particular is exact, so that the homology comparison $h$ is an isomorphism. Using \hyperref[PassesToAb]{\ref*{InducedAdjAb} (\ref*{PassesToAb})}, if $w$ is a weak equivalence in $s(\cat{C}/c)_{\ab}$, then $U_{Fc} F(w) = F U_{c}(w)$ is a weak equivalence in $s\cat{D}/Fc$. Since $U_{Fc}$ reflects weak equivalences (by \ref{ComparDiag}), $F(w)$ is a weak equivalence in $s(\cat{D}/Fc)_{\ab}$. 
\end{proof}


\begin{remark}
Consider the other factorization $\ind{F} Ab_c = Ab_{Fc} F$ of left adjoints $\cat{C}/c \ral{F} \linebreak \cat{D}/Fc \ral{Ab_{Fc}} (\cat{D}/Fc)_{\ab}$. Assume that $\cat{D}$ is an $\N$-sorted finitary variety such that the underlying $\N$-graded set has a natural structure of graded group, i.e., there is a faithful functor $\cat{D} \to \Gp^{\N}$ lifting the forgetful functor $U \colon \cat{D} \to \Set^{\N}$. Then there is a generalized Grothendieck spectral sequence \cite{Blanc92}*{Theorem 4.4} abutting to the left derived functors $L_*(Ab_{Fc} F)$ which involves homotopy operations in $s(\cat{D}/Fc)$. The map $\pi_n (\al) \colon L_n (Ab_{Fc} F)(\id_c) \to \HQ_n(Fc)$ from \ref{EffectQuilHomolC} is a bottom edge morphism in that spectral sequence.
\end{remark}

\begin{proposition} \label{EffectHQc}
Let $N$ be a module over $Fc$.
\begin{enumerate}
 \item The comparison diagram \ref{ComparDiag} induces for every degree $n \geq 0$ a natural comparison map:
\begin{equation} \label{eq:ComparHQc}
\al^* \colon \HQ^n(Fc; N) \to \HQ^n(c; \eta_c^* GN).
\end{equation}
\item If the comparison of cotangent complexes \eqref{eq:ComparCotCpxC} is a weak equivalence, then \eqref{eq:ComparHQc} is an isomorphism. This holds in particular when $F$ preserves all weak equivalences.
\item The maps $\pi_n(\al) \colon \pi_n (\ind{F} \LL_c) \to \HQ_n(Fc)$ induce a map of spectral sequences:
\[
\xymatrix{
E_2^{s,t} = \Ext^s \left( \HQ_t(Fc), N \right) \ar[d]_{\pi_t(\al)^*} \ar@{=>}[r] & \HQ^{s+t}(Fc;N) \ar[d]^{\al^*} \\
E_2^{s,t} = \Ext^s \left( \pi_t (\ind{F} \LL_c), N \right) \ar@{=>}[r] & \HQ^{s+t}(c;\eta_c^* GN) \\
}
\]
where the top row is the universal coefficient spectral sequence.  
\end{enumerate}
\end{proposition}

\begin{proof}
1. Apply the functor $\Hom_{\Mod_{Fc}}(-,N)$ to the comparison map \eqref{eq:ComparCotCpxC}:
\[
\Hom_{\Mod_{Fc}}(\LL_{Fc},N) \to \Hom_{\Mod_{Fc}}(\ind{F}(\LL_c),N) \cong \Hom_{\Mod_c}(\LL_c,\eta_c^* G N)
\]
and upon passing to cohomology, we obtain the map \eqref{eq:ComparHQc}.

2. Since $\ind{F}(\LL_c)$ and $\LL_{Fc}$ are cofibrant, a weak equivalence \eqref{eq:ComparCotCpxC} between them will induce a weak equivalence upon applying $\Hom(-,N)$, by \cite{Weibel94}*{Corollary 5.7.7}.

3. This follows from the naturality of the hyper-derived functor spectral sequence for $\Hom(-,N)$, applied to the map of chain complexes $\al \colon \ind{F} \LL_c \to \LL_{Fc}$.
\end{proof}

\subsubsection{Ground object in \texorpdfstring{$\cat{D}$}{D}}

A similar reasoning yields the following propositions.

\begin{proposition} \label{EffectQuilHomolD}
Assume the setup of \ref{ComparDiag}. Then the comparison diagram induces the following comparison maps.
\begin{enumerate}
\item A natural (up to homotopy) comparison map of cotangent complexes:
\begin{equation} \label{eq:ComparCotCpxD}
\al \colon \ep_{d \push} \ind{F} (\LL_{Gd}) \to \LL_d.
\end{equation}
\item For every degree $n \geq 0$, a natural comparison map in Quillen homology:
\begin{equation} \label{ComparQuilHomolD}
\ep_{d \push} \ind{F} \left( \HQ_n(Gd) \right) \to \HQ_n(d)
\end{equation}
which factors as a composite:
\[
\xymatrix{
\ep_{d \push} \ind{F} \left( \HQ_n(Gd) \right) \ar[r]^-{h} & L_n (\ep_{d \push} \ind{F} Ab_{Gd})(\id_{Gd}) \ar[r]^-{\pi_n (\al)} & \HQ_n(d) \\
}
\]
where $h$ is a left edge morphism in a composite spectral sequence:
\[
E^2_{s,t} = \left( L_s (\ep_{d \push} \ind{F}) \right) \left( \HQ_t(Gd) \right) \Ra L_{s+t}(\ep_{d \push} \ind{F} Ab_{Gd})(\id_{Gd}) = \pi_{s+t} (\ep_{d \push} \ind{F} \LL_{Gd}).
\]
\item If $F$ preserves pullbacks and $\ep_{d \push}$ is exact, then the map $h \colon \ep_{d \push} \ind{F} \left( \HQ_n(Gd) \right) \ral{\cong} \pi_n \ep_{d \push} \ind{F} (\LL_{Gd})$ is an isomorphism, so that the map \eqref{ComparQuilHomolD} can be identified with $\pi_n (\al)$, the effect of \eqref{eq:ComparCotCpxD} on homotopy.
\item If $F$ preserves all weak equivalences and $\ep_d$ is an isomorphism, then \eqref{eq:ComparCotCpxD} is a weak equivalence. If moreover $F$ passes to Beck modules, then \eqref{ComparQuilHomolD} is an isomorphism.
\end{enumerate}
\end{proposition}

\begin{proposition} \label{EffectHQd}
Let $N$ be a module over $d$.
\begin{enumerate}
 \item The comparison diagram \ref{ComparDiag} induces for every degree $n \geq 0$ a natural comparison map:
\begin{equation} \label{ComparHQd}
\al^* \colon \HQ^*(d; N) \to \HQ^*(Gd; GN).
\end{equation}
 \item If the comparison of cotangent complexes \eqref{eq:ComparCotCpxD} is a weak equivalence, then the map  \eqref{ComparHQd} is an isomorphism.
 \item The maps $\pi_n(\al) \colon \pi_n (\ep_{d \push} \ind{F} \LL_{Gd}) \to \HQ_n(d)$ induce a map of spectral sequences:
\[
\xymatrix{
E_2^{s,t} = \Ext^s \left( \HQ_t(d), N \right) \ar[d]_{\pi_t(\al)^*} \ar@{=>}[r] & \HQ^{s+t}(d;N) \ar[d]^{\al^*} \\
E_2^{s,t} = \Ext^s \left( \pi_t (\ep_{d \push} \ind{F} \LL_{Gd}), N \right) \ar@{=>}[r] & \HQ^{s+t}(Gd;GN) \\
}
\]
where the top row is the universal coefficient spectral sequence.  
\end{enumerate}
\end{proposition}

\section{Examples} \label{sec:Examples}

In this section, we study three examples. The first serves as a warm-up. The second relates Andr\'e-Quillen cohomology to Hochschild cohomology (Proposition \ref{HHHQ}). The third shows how Quillen cohomology of a $\Pi$-algebra with coefficients in a truncated module can be computed within the world of truncated $\Pi$-algebras (Theorem \ref{HQtrunc2}), which have a much simpler structure than (non-truncated) $\Pi$-algebras.

\subsection{Abelian groups}

Consider the functor $Ab \colon \Gp \to \Ab$ that kills commutators, i.e., $Ab(G) = G/[G,G]$, which is left adjoint to the inclusion functor $\io \colon \Ab \to \Gp$. In the notation above, the functor $Ab$ is $Ab_{\{\ast\}} \colon \Gp / \{\ast\} \to (\Gp / \{\ast\})_{\ab} \cong \Ab$, the abelianization functor over the trivial group $\{\ast\}$.

Although $Ab$ does not preserve kernel pairs in general, it does pass to Beck modules. Recall that for a (left) $G$-module $M$, the semidirect product $G \lt M$ is the group with underlying set $G \x M$ and multiplication $(g,m)(g',m') = (gg', m + gm')$.

\begin{proposition} \label{Commutativization}
$Ab \colon \Gp \to \Ab$ passes to Beck modules, on which it induces the coinvariants functor $(-)_G \colon \Mod_G \to \Ab$.
\end{proposition}

\begin{proof}
Let us first compute $Ab(G \lt M)$. Commutators in $G \lt M$ are given by
\[
\left[ (g_1,m_1),(g_2,m_2) \right] = \left( [g_1,g_2], m_1 - g_1 g_2 g_1^{-1} m_1 + g_1 m_2 - g_1 g_2 g_1^{-1} g_2^{-1} m_2 \right).
\]
Applying $Ab$ to the split extension $G \lt M \to G$ yields a split extension $Ab(G \lt M) \to Ab(G)$ in $\Ab$ whose kernel is $M$ modulo the subgroup
\begin{align*}
& \lan m_1 - g_1 g_2 g_1^{-1} m_1 + g_1 m_2 - g_1 g_2 g_1^{-1} g_2^{-1} m_2 \mid g_i \in G, m_i \in M \ran \\
&= \lan m - g m \mid g \in G, m \in M \ran.
\end{align*}
In other words, we have $Ab(G \lt M) \cong Ab(G) \op M_G$, where $M_G$ is the abelian group of coinvariants of $M$.

Moreover, $Ab \colon \Gp \to \Ab$ preserves the pullback that defines the multiplication structure map:
\[
Ab \left( (G \lt M) \x_G (G \lt M) \right) = Ab(G \lt M) \x_{Ab(G)} Ab(G \lt M).
\]
In $\Gp$ as well as in $\Ab$, we think of the module as the kernel of the split extension, and in this case, a $G$-module $M$ is sent to the abelian group $M_G$.
\end{proof}

Let us describe the effect of the adjunction $Ab \colon \Gp \rla \Ab \colon \io$ on Quillen homology. Note that the right adjoint $\io$ preserves regular epimorphisms, which are just surjections. Hence, the prolonged adjunctions are Quillen pairs.

Note also that the unit of the adjunction is $\eta_G \colon G \surj G/[G,G]$ and the counit is the identity. We work with a ground object $G$ in $\Gp$, since we get nothing new from a ground object in $\Ab$. The comparison diagram \eqref{DiagAdjHAc} becomes:
\begin{equation} \label{DiagHQgroups}
\xymatrix @R=3pc @C=3pc {
s\Gp/G \ar@<-0.6ex>[d]_{Ab} \ar@<0.6ex>[r]^-{Ab_G} & s \Mod_G \ar@<0.6ex>[l]^-{G \lt -} \ar@<-0.6ex>[d]_{(-)_G} \\
s\Ab/Ab(G) \ar@<-0.6ex>[u]_{\eta_G^* \io} \ar@<0.6ex>[r]^-{\Src} & s\Ab \ar@<0.6ex>[l]^-{Ab(G) \op -} \ar@<-0.6ex>[u]_{\Triv}
} \end{equation}
and by \ref{ComparDiag}, it prolongs to four Quillen pairs. Here $\Src$ is the ``source'' functor, which is the abelianization over any abelian group, and $\Triv$ is the functor assigning to an abelian group the trivial $G$-action. Indeed, the right adjoint on Beck modules is $\eta_G^* \io$. Given a Beck module $Ab(G) \op A$, view it as a split extension of groups, which means $A$ has a trivial $Ab(G)$ action, and then pull the action back along $\eta_G \colon G \to G/[G,G]$, which endows $A$ with the trivial $G$-action.

\begin{remark}
In \ref{Commutativization}, we checked explicitly that $Ab$ induces the functor $(-)_G$ on Beck modules. We could also look at the induced right adjoint $\eta_G^* \io = \Triv$ and use its left adjoint to complete diagram \eqref{DiagHQgroups}. The left adjoint of $\Triv = \ep^*$ is indeed $(-)_G = \ep_{\push} = \Z \ot_{\Z G} (-)$, where $\ep \colon \Z G \to \Z$ is the augmentation.
\end{remark}

We now formulate the result about Quillen homology.

\begin{proposition} \label{EffectHQgroups}
Let $C_{\bu} \to G$ be a cofibrant replacement of $G$ in groups and let $\LL_G$ denote the cotangent complex of $G$. Then the following holds:
\[
\pi_* \left( C_{\bu} / [C_{\bu},C_{\bu}] \right) = \pi_* \left( (\LL_G)_G \right).
\]
\end{proposition}

\begin{proof}
Starting from a cofibrant replacement of $G$ in $\Gp$ (or equivalently, of $\id_G$ in $\Gp/G$) in the upper left corner of \eqref{DiagHQgroups}, going down then right yields
\begin{align*}
\Src \circ Ab(C_{\bu} \to G) &= \Src \left( Ab(C_{\bu}) \to Ab(G) \right) \\
&= Ab(C_{\bu}) = C_{\bu} / [C_{\bu},C_{\bu}]
\end{align*}
whereas going right then down yields $\left( Ab_G (C_{\bu} \to G) \right)_G = (\LL_G)_G$. Taking $\pi_*$ gives a well-defined equality, since the simplicial $G$-module $\LL_G$ is defined up to homotopy.
\end{proof}

In fact, one can compute both sides explicitly and check that they coincide. For groups, abelianization is $Ab_G G = I_G = \ker(\Z G \to \Z)$ and the cotangent complex is discrete, meaning $\LL_G \to I_G$ is a cofibrant replacement, in particular a flat resolution. Taking coinvariants results in the derived functors thereof, namely, group homology:
\[
\pi_* \left( (\LL_G)_G \right) = L_*(-)_G (I_G) = \Hy_*(G; I_G).
\]
Using the short exact sequence $0 \to I_G \to \Z G \to \Z \to 0$ of $G$-modules, the connecting morphism $\Hy_{i+1}(G; \Z) \to \Hy_i(G; I_G)$ is an isomorphism for all $i \geq 0$, from which we conclude $\pi_i \left( (\LL_G)_G \right) = \Hy_{i+1}(G;\Z)$ for all $i \geq 0$. On the other hand, \cite{Goerss07}*{Example 4.26} uses a different argument to show $\pi_i \left( C_{\bu} / [C_{\bu},C_{\bu}] \right) = \Hy_{i+1}(G;\Z)$ for all $i \geq 0$. Proposition \ref{EffectHQgroups} is consistent with these computations.

\subsection{Commutative algebras}

Let $R$ be a fixed commutative ring; denote by $\Alg_R$ the category of associative $R$-algebras and by $\Com_R$ the category of commutative $R$-algebras. (All our rings and algebras are assumed associative and unital.) Consider the functor $Com \colon \Alg_R \to \Com_R$ which kills the 2-sided ideal generated by commutators, that is, $Com(A) = A / [A,A]$. It is left adjoint to the inclusion functor $\io \colon \Com_R \to \Alg_R$, which preserves regular epimorphisms (i.e., surjections).

Recall that Beck modules over an associative $R$-algebra $A$ are $A$-bimodules over $R$, meaning that scalars in $R$ act the same way on the left and the right; we denote this category by $\Bimod{A}_R$. Beck modules over a commutative $R$-algebra $A$ are $A$-modules in the usual sense, which we denote $\Modd{A}$.

\begin{proposition} \label{InducedCom}
\begin{enumerate}
\item The functor $Com \colon \Alg_R \to \Com_R$ passes to Beck modules.
\item It induces the ``central quotient'' functor $CQ \colon \Bimod{A}_R \to \Modd{Com(A)}$ which coequalizes the two actions.
\end{enumerate}
\end{proposition}

\begin{proof}
Start with a Beck module over $A$ in $\Alg_R$, i.e., a split extension $p \colon A \op M \to A$ satisfying $M^2 = 0$. Applying $Com$ to it yields a split extension 
\[
\xymatrix{
0 \ar[r] & K \ar[r] & Com(A \op M) \ar@<0.6ex>[r]^-{Com(p)} & Com(A) \ar@<0.6ex>[l]^-{Com(s)} \ar[r] & 0 \\
}
\]
in $\Com_R$. It remains to show that its kernel has square zero.

\textbf{Commutators in $A \op M$.} Using the decomposition $(a,m) = (a,0)+(0,m)$, commutators will be generated by those of the forms $[(a,0),(a',0)] = ([a,a'],0)$ and $[(a,0),(0,m)] = (0, a \cdot m - m \cdot a)$. Thus the kernel is
\begin{equation} \label{CentralQuotient}
K \simeq M / \lan a \cdot m - m \cdot a \ran
\end{equation}
where we kill the $A$-subbimodule generated by all elements of that form.

\textbf{$K$ has square zero.} Take two elements $x,x' \in K = \ker Com(p) \subset Com(A \op M)$ and choose representatives $(c,m)$ and $(c',m')$ in $A \op M$, for $c,c' \in [A,A]$. Then $xx'$ is represented by $(c,m)(c',m') = (cc', c \cdot m' + m \cdot c')$. One readily checks that elements of the form $c \cdot m$ and $m \cdot c$ are zero in $Com(A \op M)$, for any $m \in M$ and $c \in [A,A]$. This proves the first assertion, and \eqref{CentralQuotient} proves the second.
\end{proof}

The adjunction $Com \colon \Alg_R \rla \Com_R \colon \io$ allows us to compare the two categories. According to \ref{InducedCom}, the comparison diagram \eqref{DiagAdjHAc} becomes
\begin{equation} \label{DiagHAalg}
\xymatrix @R=3pc @C=3pc {
\Alg_R/A \ar@<-0.6ex>[d]_{Com} \ar@<0.6ex>[r]^-{A \ot I_{(-)} \ot A} & \Bimod{A}_R \ar@<0.6ex>[l]^-{A \op -} \ar@<-0.6ex>[d]_{CQ} \\
\Com_R/Com(A) \ar@<-0.6ex>[u]_{\eta_A^* \io} \ar@<0.6ex>[r]^-{Com(A) \ot \Om_{(-)/R}} & \Modd{Com(A)} \ar@<0.6ex>[l]^-{Com(A) \op -} \ar@<-0.6ex>[u]_{\text{same action}} \\
}
\end{equation}
where ``same action'', the right adjoint on the right, means that we view a $Com(A)$-module as an $A$-bimodule by acting via the unit $A \to Com(A) = A / [A,A]$ both on the left and the right. Abelianization in associative algebras is $Ab_A(B \to A) = A \ot_B I_B \ot_B A$ where $I_B$ denotes the kernel of the multiplication map $m \colon B \ot_R B \to B$. Abelianization in commutative algebras is $Ab_S(T \to S) = S \ot_T \Om_{T/R}$ where $\Om_{T/R}$ denotes the module of K\"ahler differentials $I_T / I_T^2$. By \ref{ComparDiag}, diagram \eqref{DiagHAalg} prolongs to four Quillen pairs.

\begin{remark}
One can view $\Bimod{A}_R$ as the category of left $A \ot_R A^{\opp}$ modules, and the ``same action'' functor $\Modd{Com(A)} \to \Bimod{A}_R$ as the restriction $(\eta_A m)^*$ along $A \ot_R A^{\opp} \ral{m} A \ral{\eta_A} Com(A)$. Its left adjoint is the pushforward $(\eta_A m)_{\push} = Com(A) \ot_{A \ot_R A^{\opp}} -$ which is indeed the functor coequalizing the two actions.
\end{remark}

Some special cases are of particular interest. When the $R$-algebra $A$ is just $R$ itself -- and is in particular commutative -- the comparison diagram \eqref{DiagHAalg} becomes:
\[
\xymatrix @R=3pc @C=3pc {
\Alg_R/R \ar@<-0.6ex>[d]_{Com} \ar@<0.6ex>[r]^-{R \ot I_{(-)} \ot R} & \Bimod{R}_R \ar@<0.6ex>[l]^-{R \op -} \ar@<-0.6ex>[d]_{\id} \\
\Com_R/R \ar@<-0.6ex>[u]_{\io} \ar@<0.6ex>[r]^{R \ot \Om_{(-)/R}} & \Modd{R}. \ar@<0.6ex>[l]^-{R \op -} \ar@<-0.6ex>[u]_{\id}
}
\]
The diagram says that killing all products can be done in two steps, by killing all commutators first. One could try to use the Grothendieck composite spectral sequence for the non-abelian setting \cite{Blanc92}*{Theorem 4.4} to relate Quillen homology in $\Alg_R$ to Quillen homology in $\Com_R$, i.e., Andr\'e-Quillen homology. This approach requires the knowledge of homotopy operations in $s\Com_R$, which are known notably for $R = \F_2$ \cite{Dwyer80} \cite{Goerss90}*{Chapter II} \cite{Goerss95}.

More generally, another interesting case is when the cotangent complex in associative algebras is discrete, i.e., $\LL_A \to Ab_A A$ is a weak equivalence. Quillen \cite{Quillen70}*{Proposition 3.6} shows that this happens under the condition $\Tor_i^R(A,A) = 0$ for all $i \geq 1$ (for example if $R$ is a field), in which case $\HA^*(A;M) \cong \HQ^*(A;M)$ is essentially the same as the usual Hochschild cohomology, and likewise for homology.

\begin{proposition} \label{HHHQ}
Let $A$ be a commutative $R$-algebra satisfying $\Tor_i^R(A,A) = 0$ for all $i \geq 1$. Then for every $j \geq 1$, the Hochschild homology of $A$ can be written as
\[
\HH_{j+1}(A) = \pi_j \left( A \ot_{Com(C_{\bu})} \Om_{Com(C_{\bu})/R} \right)
\]
where $C_{\bu} \to A$ is a cofibrant replacement of $A$ in $\Alg_R$. In particular, there is a comparison map $\HH_{j+1}(A) \to \HQ_j(A)$ for $j \geq 1$.
\end{proposition}

\begin{proof}
Starting from a cofibrant replacement $C_{\bu} \to A$ in $\Alg_R$ and going right in \eqref{DiagHAalg}, one obtains $\LL_A \to I_A$, which is a weak equivalence because of the flatness assumption on $A$. Then going down yields $A \ot_{A \ot_R A^{\opp}} \LL_A$, whose $\pi_*$ is $\Tor_*^{A \ot_R A^{\opp}}(A,I_A)$. Again by the flatness assumption, Hochschild homology $\HH_*(A)$ is not just a relative $\Tor$ but the (absolute) $\Tor_*^{A \ot_R A^{\opp}}(A,A)$. The short exact sequence of bimodules $0 \to I_A \to A \ot_R A^{\opp} \to A \to 0$ gives a natural isomorphism $\Tor_{i+1}^{A \ot_R A^{\opp}}(A,A) \cong \Tor_i^{A \ot_R A^{\opp}}(A,I_A)$ for all $i \geq 1$.

On the other hand, going down in the diagram yields $Com(C_{\bu}) \to A$ and then going right yields $A \ot_{Com(C_{\bu})} \Om_{Com(C_{\bu})/R}$. The comparison map is $\pi_*$ of \eqref{eq:ComparCotCpxD}, which measures the failure of $Com \colon \Alg_R \to \Com_R$ to preserve weak equivalences.
\end{proof}

\begin{remark}
The comparison map in \ref{HHHQ} is an edge morphism in a spectral sequence of Quillen \cite{Quillen70}*{Theorem 8.1}.
\end{remark}

\subsection{Truncated \texorpdfstring{$\Pi$}{Pi}-algebras} \label{sec:TrunPiAlg}

A $\Pi$-algebra is the algebraic structure describing the action of the primary homotopy operations on the homotopy groups of a pointed space $X$. More details can be found in \cite{Blanc04}*{\S 4} \cite{Stover90}*{\S 4}; we recall the essentials.

Let $\PI$ denote the homotopy category of finite (possibly empty) wedges of spheres $\We S^{n_i}$, with $n_i \geq 1$. The category $\PI$ has finite coproducts, given by the wedge, and thus its opposite $\PI^{\opp}$ is an algebraic theory as in Definition \ref{def:AlgTh}.

\begin{definition} \label{def:PiAlg}
A \Def{$\PI$-algebra} is a contravariant functor $A \colon \PI \to \Set$ that sends wedges to products, i.e., a product-preserving functor $\PI^{\opp} \to \Set$ (or equivalently to pointed sets). In other words, a $\PI$-algebra is a model for the algebraic theory $\PI^{\opp}$.

Let $\PiAlg \dfn \Model(\PI^{\opp})$ denote the category of $\Pi$-algebras.
\end{definition}

The prototypical example is the functor $[-,X]$ of pointed homotopy classes of pointed maps into $X$, called the homotopy $\Pi$-algebra of the pointed space $X$. A $\Pi$-algebra $A$ can be viewed as a graded group $\{\pi_i = A(S^i)\}$ (abelian for $i \geq 2$) equipped with primary homotopy operations induced by maps between wedges of spheres, such as precomposition operations $\al^* \colon \pi_k \to \pi_n$ for every $\al \in \pi_n(S^k)$. The additional structure is determined by operations of that form, Whitehead products, and the $\pi_1$-action on higher $\pi_i$, and there are classical relations between them.

\renewcommand\thesubsubsection{\thesubsection.\arabic{subsubsection}}
\subsubsection{Postnikov truncation of \texorpdfstring{$\Pi$}{Pi}-algebras}

\begin{definition}
A $\Pi$-algebra $A$ is called \Def{$n$-truncated} if for all $i > n$, we have $A(S^i) = \ast$, the trivial pointed set.
\end{definition}

Denote by $\PiAlg_1^n$ the full subcategory of $\PiAlg$ consisting of $n$-truncated $\Pi$-algebras. Denote by $\PI_n$ the full subcategory of $\PI$ consisting of wedges of spheres of dimension at most $n$, and let $I_n \colon \PI_n \to \PI$ be the inclusion functor. One can go the other way, by removing spheres above a certain dimension. To make this precise, assume without loss of generality that the category $\PI$ is skeletal, i.e., each isomorphism class contains exactly one object. In other words, we choose a representative space for each homotopy type $\We S^{n_i}$ in $\PI$. Define a ``truncation'' functor $T_n \colon \PI \to \PI_n$ by $T_n \left( \We_{i=1}^k S^{n_i} \right) = \We_{n_i \leq n} S^{n_i}$. This functor sends a map $f \colon \We_i S^{n_i} \to \We_j S^{m_j}$ to the homotopy lift:
\[
\xymatrix{
\We_{n_i \leq n} S^{n_i} \ar@{-->}[drr]_{T_n f} \ar@{^{(}->}[r] & \We_i S^{n_i} \ar[r]^f & \We_j S^{m_j} \\
& & \We_{m_j \leq n} S^{m_j} \ar@{^{(}->}[u] \\
}
\]
which exists and is unique, since $\We_{m_j \leq n} S^{m_j} \inj \We_j S^{m_j}$ is an isomorphism on $\pi_k$ for $k \leq n$. By the same argument, $I_n$ is left adjoint to $T_n$. The unit $1 \to T_n I_n$ is the identity, and the counit $I_n T_n \to 1$ is the inclusion of wedge summands of small dimension.

Both $I_n \colon \PI_n \to \PI$ and $T_n \colon \PI \to \PI_n$ preserve coproducts (wedges), and thus induce restriction functors on models $I_n^* \colon \Model(\PI^{\opp}) \to \Model(\PI_n^{\opp})$ and $T_n^* \colon \Model(\PI_n^{\opp}) \to \Model(\PI^{\opp})$, where $T_n^*$ lands in the subcategory $\PiAlg_1^n$.

\begin{proposition}
The functors $I_n^* \colon \PiAlg_1^n \cong \Model(\PI_n^{\opp}) \colon T_n^*$ form an equivalence of categories.
\end{proposition}

\begin{proof}
If $F$ is a product-preserving functor $\PI_n^{\opp} \to \Set$, then we have $I_n^* T_n^* F = (T_n I_n)^* F = F$, since $T_n I_n$ is the identity. On the other hand, if $A$ is an $n$-truncated $\Pi$-algebra, we have $T_n^* I_n^* A = (I_n T_n)^*A \cong A$. Indeed, $A$ sends all counit maps:
\[
I_n T_n(\We_i S^{n_i}) = \We_{n_i \leq n} S^{n_i} \inj \We_i S^{n_i}
\]
to isomorphisms since $A$ is $n$-truncated.
\end{proof}

\begin{example}
$1$-truncated $\Pi$-algebras are the same as groups. In the category $\PI_1$, the hom-set $\left[ S^1, \We_{i=1}^k S^1 \right] = \pi_1 \left( \We_{i=1}^k S^1 \right)$ is a free group on $k$ generators. This yields an equivalence of algebraic theories $\PI_1^{\opp} \cong \cat{T}_{\Gp}$ and thus an equivalence of models $\PiAlg_1^1 \cong \Gp$.
\end{example}


Write $P_n \colon \PiAlg \to \PiAlg_1^n$ for $I_n^*$, which is the Postnikov $n$-truncation of $\Pi$-algebras, and $\io_n \colon \PiAlg_1^n \to \PiAlg$ for $T_n^*$, which is the inclusion of $n$-truncated $\Pi$-algebras.

\begin{proposition}
$P_n$ is left adjoint to $\io_n$.
\end{proposition}

\begin{proof}
\textit{(Functor point of view)} $I_n \colon \PI_n \to \PI$ is the left adjoint, and thus $I_n \colon \PI_n^{\opp} \to \PI^{\opp}$ is the right adjoint. Note that $\Fun(-,\Set)$ is a (strict) $2$-functor $\Cat^{\opp} \to \Cat$, where the superscript in $\Cat^{\opp}$ means that $1$-cells have been reversed but $2$-cells do not change. The same holds for $\Fun^{\x}(-,\Set)$, as long as we take only categories and product-preserving functors between them. Therefore $P_n = I_n^*$ is left adjoint to $\io_n = T_n^*$.
\end{proof}

\begin{proof}
\textit{(Graded group point of view)} A map $f \colon A \to \io_n B$ of $\Pi$-algebras into an $n$-truncated $\Pi$-algebra is determined by the map of graded group up to degree $n$. The additional conditions are that $f$ respect the additional structure ($\pi_1$-action, Whitehead products, and precomposition operations). These operations preserve or increase degree, which means that all the conditions coming from or landing in degree greater than $n$ are vacuous. In other words, the data of a map $f$ is the same data as the corresponding map $P_n A \to B$ in $\PiAlg_1^n$.
\end{proof}

Both $\PiAlg$ and $\PiAlg_1^n$ are categories of universal algebras -- many-sorted finitary varieties, to be more precise. The free $\Pi$-algebra on a graded set $\{X_i\}$ is $F \{X_i\} = \pi_* ( \We_i \We_{j \in X_i} S^i )$.  By combining the two adjunctions:
\[ 
\xymatrix{
\GrSet \ar@<0.6ex>[r]^F & \PiAlg \ar@<0.6ex>[l]^U \ar@<0.6ex>[r]^{P_n} & \PiAlg_1^n \ar@<0.6ex>[l]^{\io_n} \\
}
\]
we see that the free $n$-truncated $\Pi$-algebra on $\{X_i\}$ is:
\[
F_n \{X_i\} = P_n \pi_* ( \We_i \We_{j \in X_i} S^i ) = \pi_* ( P_n \We_i \We_{j \in X_i} S^i ).
\]
In both categories, projective objects are retracts of free objects and regular epimorphisms are surjections of underlying graded sets \cite{Quillen67}*{II.4, Remark 1 after Proposition 1}. In particular, the left adjoint $P_n$ preserves projectives and prolongs to a left Quillen functor. Note that $\left\{ \pi_*(P_n S^1), \pi_*(P_n S^2), \ldots, \pi_*(P_n S^n) \right\}$ is a set of finitely presentable projective generators for $\PiAlg_1^n$.

\subsubsection{Standard model structure}

The standard model structure on the category $s\PiAlg$ of simplicial $\Pi$-algebras is described in \cite{Blanc04}*{\S 4.5} and the same description holds for $s\PiAlg_1^n$. A map $f \colon X_{\bu} \to Y_{\bu}$ is a fibration (resp. weak equivalence) if it is so at the level of underlying graded sets or graded groups. Cofibrations are maps with the left lifting property with respect to trivial fibrations and can be characterized as retracts of free maps.

\begin{proposition} \label{PnFib}
The left Quillen functor $P_n \colon s\PiAlg \to s\PiAlg_1^n$ preserves weak equivalences and fibrations. In particular, it preserves cofibrant replacements.
\end{proposition}

\begin{proof}
\textit{(Functor point of view)} Let $f \colon X_{\bu} \to Y_{\bu}$ be a fibration (resp. weak equivalence) in $s\PiAlg$. Let $P$ be a projective of $\PiAlg_1^n$, exhibited as a retract of a free by $P \ral{s} F \ral{p} P$. Then $(P_n f)_* \colon \Hom(P, P_n X_{\bu}) \to \Hom(P, P_n Y_{\bu})$ is a retract of $\Hom(F,P_n f)$ so it suffices that the latter be a fibration (resp. weak equivalence) of simplicial sets.

Note that $F = F_n(S)$ is free on a graded set $S$ empty above dimension $n$, so we have:
\begin{align*}
 \Hom_{\PiAlg_1^n}(F, P_n X_{\bu}) &= \Hom_{\GrSet}(S, U P_n X_{\bu} ) \\
 &= \Hom_{\GrSet}(S, U X_{\bu} ) \\
 &= \Hom_{\PiAlg} (F(S), X_{\bu} ).
\end{align*}
Using this, we obtain:
\[
\xymatrix{
\Hom_{\PiAlg_1^n}(F, P_n X_{\bu}) \ar[d]_{\cong} \ar[r]^{(P_n f)_*} & \Hom_{\PiAlg_1^n}(F, P_n Y_{\bu}) \ar[d]^{\cong} \\
\Hom_{\PiAlg}(F(S), X_{\bu}) \ar[r]^{f_*} & \Hom_{\PiAlg}(F(S), Y_{\bu}). \\
}
\]
Since $f$ is a fibration (resp. weak equivalence) in $s\PiAlg$, the bottom row is a fibration (resp. weak equivalence) of simplicial sets.
\end{proof}

\begin{proof}
\textit{(Graded group point of view)} The map $f \colon X_{\bu} \to Y_{\bu}$ is a fibration (resp. weak equivalence) of simplicial sets in each degree, hence the map $P_n f$ is a fibration (resp. weak equivalence) of simplicial sets in each degree, that is in degrees $1$ through $n$.
\end{proof}

\begin{corollary} \label{PnCotCpx}
\begin{enumerate}
\item For every $\Pi$-algebra $A$, the comparison map of cotangent complexes $P_n (\LL_A) \ral{\sim} \LL_{P_n A}$ induced by the adjunction $P_n \dashv \io_n$ is a weak equivalence.
\item If $N$ is a module over $P_n A$, then the comparison map in Quillen cohomology:
\begin{equation} \label{HQtrunc}
\HQ^*_{\PiAlg_1^n}(P_n A; N) \ral{\cong} \HQ^*_{\PiAlg}(A; \eta_A^* \io_n N)
\end{equation}
is a natural isomorphism.
\end{enumerate}
\end{corollary}

\begin{proof}
By \ref{EffectQuilHomolC}, \ref{EffectHQc}, and \ref{PnFib}.
\end{proof}

Here $\eta_A \colon A \to \io_n P_n A$ is the Postnikov truncation map. We would like a better description of the module $\eta_A^* \io_n N$ in \eqref{HQtrunc}. Think of a module over $A$ as an abelian $\Pi$-algebra on which $A$ acts (cf. \cite{Blanc04}*{\S 4.11}), namely the kernel of the split extension as opposed to its ``total space''.

\begin{lemma}
The category $\Mod_{P_n A}$ of modules over $P_n A$ is equivalent to the full subcategory $\Mod_A^{n\text{-tr}}$ of $\Mod_A$ of modules that are $n$-truncated.
\end{lemma}

\begin{proof}
Consider the adjunction on modules:
\[
\xymatrix{
\Mod_A \ar@<0.6ex>[r]^-{P_n} & \Mod_{P_n A} \ar@<0.6ex>[l]^-{\eta_A^* \io_n}
}
\]
from \ref{AdjBeckMod}. The composite $P_n \eta_A^* \io_n$ is naturally isomorphic to the identity. Moreover, $\eta_A^* \io_n$ lands in $\Mod_A^{n\text{-tr}}$. By restricting $P_n$ to the latter, we obtain an adjunction $\Mod_A^{n\text{-tr}} \rla \Mod_{P_n A}$ where both composites $P_n \eta_A^* \io_n$ and $\eta_A^* \io_n P_n$ are naturally isomorphic to the identity.
\end{proof}

The lemma justifies the abuse of notation in the following repackaged statement.

\begin{theorem}[Truncation isomorphism] \label{HQtrunc2}
Let $A$ be a $\Pi$-algebra and $N$ a module over $A$ that is $n$-truncated. Then there is a natural isomorphism:
\[
\HQ^*_{\PiAlg_1^n}(P_n A; N) \ral{\cong} \HQ^*_{\PiAlg}(A; N).
\]
\end{theorem}

The following example is of interest in light of \cite{Blanc04}*{Theorems 1.3 and 9.6}.

\begin{example} \label{HQloops}
Let $A$ be an $n$-truncated $\Pi$-algebra. For $k$ a positive integer, the $k$-fold loops $\Om^k A$ form a module over $A$ (which is zero if $k \geq n$) and we are interested in the cohomology groups $\HQ^*(A; \Om^k A)$. Since $\Om^k A$ is $(n-k)$-truncated, Theorem \ref{HQtrunc2} gives a natural isomorphism $\HQ^*_{\PiAlg_1^{n-k}}(P_{n-k} A; \Om^k A) \cong \HQ^*_{\PiAlg}(A; \Om^k A)$.
\end{example}

\end{document}